\documentclass[10pt]{amsart} 

\usepackage{fullpage}
\usepackage{url,amssymb,amsmath,amsfonts,amsthm,mathrsfs}

\usepackage[all]{xy} \SelectTips{cm}{10}

\usepackage{lipsum} 

\usepackage{color}

\definecolor{webcolor}{rgb}{0.8,0,0.2}
\definecolor{webbrown}{rgb}{.6,0,0}
\usepackage[
        colorlinks,
        linkcolor=webbrown,  filecolor=webcolor,  citecolor=webbrown, 
        backref,
        pdfauthor={David Zywina}, 
       pdftitle={},
]{hyperref}
\usepackage[alphabetic,backrefs,lite]{amsrefs} % for bibliography

\numberwithin{equation}{section}

\renewcommand{\AA}{\mathbb A}

\newcommand{\CC}{\mathbb C}

\newcommand{\NN}{\mathbb N}
\newcommand{\PP}{\mathbb P}
\newcommand{\QQ}{\mathbb Q}
\newcommand{\RR}{\mathbb R}

\newcommand{\ZZ}{\mathbb Z}

\newcommand{\OO}{\mathcal O}
  \newcommand{\calF}{\mathcal F}
\newcommand{\calH}{\mathcal H}

\newcommand{\calB}{\mathcal B}
\newcommand{\calS}{\mathcal S}

\newcommand{\calC}{\mathcal C}

\newcommand{\calX}{\mathcal X}

\newcommand{\scrL}{\mathscr L}

\newcommand{\norm}[1]{ \left|\!\left| #1 \right|\!\right|  }

\def\Spec{\operatorname{Spec}}

\def\Gal{\operatorname{Gal}}
\def\ord{\operatorname{ord}} 
\def \GL {\operatorname{GL}}

\def \SL {\operatorname{SL}}

\def\Aut{\operatorname{Aut}}

\newcommand{\defi}[1]{\textsf{#1}} % for defined terms

\newcommand\blank[1]{}

\def\bbar#1{\setbox0=\hbox{$#1$}\dimen0=.2\ht0 \kern\dimen0 
\overline{\kern-\dimen0 #1}}
 
\newcommand{\Kbar}{{\bbar{K}}} 
\newcommand{\Lbar}{\bbar{L}}

\newtheorem{thm}{Theorem}[section]
\newtheorem{lemma}[thm]{Lemma}

\newtheorem{prop}[thm]{Proposition}

\theoremstyle{definition}

\theoremstyle{remark}
\newtheorem{remark}[thm]{Remark}

\newenvironment{romanenum}{\hfill \begin{enumerate} }{\end{enumerate}}
\newenvironment{alphenum}{\hfill \begin{enumerate} }{\end{enumerate}}

\begin{document}

\title{Improved bounds for integral points on modular curves using Runge's method}
\subjclass[2020]{Primary 11G18; Secondary 11F11}

% 11G18 Arithmetic aspects of modular and Shimura varieties
% 11F11 Holomorphic modular forms of integral weight
%% MSC-class: 11G18 (Primary) 11F80 (Secondary)
%%\keywords{}
\author{David Zywina}
\address{Department of Mathematics, Cornell University, Ithaca, NY 14853, USA}
\email{zywina@math.cornell.edu}
%\urladdr{http://pi.math.cornell.edu/~zywina}

\begin{abstract}
Consider a modular curve $X_G$ defined over a number field $K$, where $G$ is a subgroup of $\GL_2(\ZZ/N\ZZ)$ with $N>2$.  The curve $X_G$ comes with a morphism $j\colon X_G\to \PP^1_K=\AA^1_K \cup\{\infty\}$ to the $j$-line.   For a finite set of places $S$ of $K$ that satisfies a certain condition, Runge's method shows that there are only finitely many points $P \in X_G(K)$ for which $j(P)$ lies in the ring $\OO_{K,S}$ of $S$-units of $K$.   We prove an explicit version which shows that if $j(P)\in \OO_{K,S}$ for some $P\in X_G(K)$, then the absolute logarithmic height of $j(P)$ is bounded above by $N^{12} \log N$.   Explicits upper bounds have already been obtained by Bilu and Parent though they are not polynomial in $N$.  The modular functions needed to apply Runge's method are constructing using Eisenstein series of weight $1$.
\end{abstract}

\maketitle

\section{Introduction}

Fix an integer $N>2$ and consider a subgroup $G$ of $\GL_2(\ZZ/N\ZZ)$ that contains $-I$.  Associated to the group $G$ is a \defi{modular curve} $X_G$ that is defined over the number field $K_G:=\QQ(\zeta_N)^{\det(G)}$ and is smooth, projective and geometrically irreducible; see \S\ref{S:first modular curve} for details.   In particular, the curve $X_G$ is defined over $\QQ$ when $\det(G)=(\ZZ/N\ZZ)^\times$. The curve $X_G$ comes with a nonconstant morphism \[
j\colon X_G\to \PP^1_{K_G}=\AA^1_{K_G} \cup \{\infty\}
\]
to the $j$-line.  The \defi{cusps} of $X_G$ are the points lying over $\infty$.  We define $Y_G$ to be the open subvariety of $X_G$ that is the complement of its cusps.  

Fix a number field $K\supseteq K_G$.  Let $S$ be a finite set of places of $K$ that contains all the infinite places and let $\OO_{K,S}$ be the ring of $S$-integers in $K$.   Let $\mathfrak{c}_{G,K}$ be the number of orbits of the action of the Galois group $\Gal(\Kbar/K)$ on the set of cusps $X_G(\Kbar)-Y_G(\Kbar)$.   We say that the pair $(K,S)$ satisfies \defi{Runge's condition for $X_G$} if $|S|< \mathfrak{c}_{G,K}$.

Assume that $(K,S)$ satisfies Runge's condition for $X_G$.  \emph{Runge's method} shows that there only finitely many points $P\in Y_G(K)$ with $j(P)\in \OO_{K,S}$.    For background on Runge's method see \cite[Chapter 5]{MR2459823}, \cite[\S9.6.5]{MR2216774} or \cite[\S4]{MR3917917}.   An effective version for modular curves was given by Bilu and Parent \cite{BiluParent} where they showed that for all points $P\in Y_G(K)$ with $j(P)\in \OO_{K,S}$, we have
\begin{align} \label{E:BP}
h(j(P)) \leq 36 |S|^{|S|/2+1}(N^2 |G|/2)^{|S|} \log(2N),
\end{align}
where $h$ denotes the logarithmic absolute height.   

Our main result gives a bound  for $h(j(P))$ which is polynomial in $N$ and independent of $|S|$; the bound in (\ref{E:BP}) is polynomial in $N$ only when we bound $|S|$.    

\begin{thm} \label{T:main}
Fix an integer $N>2$ and a subgroup $G$ of $\GL_2(\ZZ/N\ZZ)$ containing $-I$.  
Let $\mu$ be the degree of the morphism $j\colon X_G\to \PP^1_{K_G}$.  
Assume that $(K,S)$ satisfies Runge's condition for $X_G$, where $K\supseteq K_G$ is a number field and $S$ is a finite set of places of $K$ that contains all the infinite places.    Then for any point $P\in Y_G(K)$ with $j(P)\in \OO_{K,S}$, we have
$h(j(P))\leq  4 (\mu+4)^4 \log N$
and hence
\[
h(j(P))\leq  N^{12} \log N.
\] 
\end{thm}

\begin{remark}
\begin{romanenum}
\item
Since $G$ contains $-I$, we have $\mu=[\SL_2(\ZZ/N\ZZ): G\cap \SL_2(\ZZ/N\ZZ)]$.  In particular, we have $\mu \leq |\SL_2(\ZZ/N\ZZ)/\{\pm I\}| \leq N^3/2$.  Using that $N>2$, we find that $4(\mu+4)^4\leq 4(N^3/2+4)^4 \leq N^{12}$.  This explains how the first inequality of Theorem~\ref{T:main} implies the second one.
\item
Since $h$ is the logarithmic absolute height, Theorem~\ref{T:main} implies that there are only finitely many points $P\in Y_G(K)$ with $j(P)\in \OO_{K,S}$ as we vary over all pairs $(K,S)$ satisfying Runge's condition for $X_G$, where $K$ is a subfield of a fixed algebraic closure of $K_G$.  
\item
Consider a subgroup $G$ of $\GL_2(\ZZ/N\ZZ)$ for which $X_G$ has at least $3$ distinct cusps.   Take any number field $K \supseteq K_G$ and any finite set of places $S$ of $K$.  From a classical theorem of Siegel, there are only finitely many points $P\in X_G(K)$ for which $j(P)\in \OO_{K,S}$.  Unfortunately, Siegel's theorem gives no way to bound the height of the $j$-invariants $j(P)$ that arise.   Bilu \cite[\S5]{MR1336325} showed that the heights could actually be effectively bounded by making use of Baker's method.   A quantitative version was given by Sha \cite{MR3250041} and see also \cite{MR4450724}.   We will not state the explicit bound of Sha, but simply remark that it implies that $\log(h(j(P))+1) \leq C_{K,S} N \log N$ holds for all $P\in Y_G(K)$ with $j(P)\in \OO_{K,S}$, where $C_{K,S}$ is a positive constant depending only on the pair $(K,S)$.    Note that when Runge's condition holds, the bounds in Theorem~\ref{T:main} will be stronger and will have no dependency on $(K,S)$.

\item
The condition $N>2$ is used several times during the proof; for such $N$, the classical modular curve $X(N)$ over $\QQ(\zeta_N)$ is a fine moduli space.  For the excluded case $N=2$, one can simply lift $G$ to a subgroup of $\GL_2(\ZZ/4\ZZ)$ to obtain bounds.

\item
Let $g$ be the genus of $X_G$.   We have $\mu < 101(g+1)$, see the comments after \cite[Proposition~2.3]{MR2016709}.   For a pair $(K,S)$ that satisfies Runge's condition for $X_G$, Theorem~\ref{T:main} implies that
\[
h(j(P)) \leq 4 ((101(g+1)+4)^4 \cdot \log N
\]
holds for all $P\in Y_G(K)$ with $j(P)\in \OO_{K,S}$.

\item
The points of our modular curves give important arithmetic information about elliptic curves which we now recall; this will not be used elsewhere.   Fix a number field $K\supseteq K_G$.  Let $E$ be an elliptic curve over $K$ with $j$-invariant $j(E) \notin \{0,1728\}$.   The $N$-torsion subgroup $E[N]$ of $E(\Kbar)$ is a free $\ZZ/N\ZZ$-module of rank $2$.   The natural Galois action on $E[N]$ and a choice of basis gives a Galois representation $\rho_{E,N}\colon \Gal(\Kbar/K) \to \Aut_{\ZZ/N\ZZ}(E[N])\cong \GL_2(\ZZ/N\ZZ)$.    One can show that $\rho_{E,N}(\Gal(\Kbar/K))$ is conjugate in $\GL_2(\ZZ/N\ZZ)$ to a subgroup of $G^t$ if and only if $j(E)=j(P)$ for some $P\in Y_G(K)$.   Here $G^t$ is the group obtained by taking the transpose of the elements of $G$.    As a warning, we note that in the literature, our modular curve $X_G$ is sometimes denoted $X_{G^t}$.
\end{romanenum}
\end{remark}

\subsection{Overview} \label{SS:overview}

Fix an integer $N>2$ and a subgroup $G$ of $\GL_2(\ZZ/N\ZZ)$.  Fix a number field $K\supseteq K_G$  and a finite set $S$ of places of $K$ containing all the infinite places.  The set $\calC_G:=X_G(\Kbar)-Y_G(\Kbar)$ of cusps has a natural action by $\Gal_K:=\Gal(\Kbar/K)$.  Assume that Runge's condition for $X_G$ holds for the pair $(K,S)$, i.e., the number of $\Gal_K$-orbits on $\calC_G$ is strictly greater than $|S|$.

Consider any nonempty and proper subset $\Sigma\subseteq \calC_G$ that is $\Gal_K$-stable.  To apply Runge's method, we need a nonconstant function $\varphi \in K(X_G)$ that satisfies the following properties:
\begin{itemize}
\item The poles of $\varphi$ occur only at cusps of $X_G$.
\item The function $\varphi$ is regular at each cusp in $\Sigma$.   
\item  The function $\varphi$ is a root of a monic polynomial with coefficients in $\ZZ[j]$.
\end{itemize}
The existence of a nonconstant function $\varphi \in K(X_G)$ that satisfies the first two properties is an easy consequence of the Riemann--Roch theorem.   Since such a function $\varphi$ is integral over $\QQ[j]$, we obtain the third property after scaling $\varphi$ by an appropriate positive integer.  However, this existence is not enough for our application since we also need good estimates on the function $\varphi$; especially near the cusps.

The functions $\varphi$ used by Bilu and Parent in \cite{BiluParent} are \emph{modular units}, i.e., their zeros and poles only occur at cusps.  The existence of suitable modular units is not a consequence of the Riemann--Roch theorem but follows from the Manin--Drinfeld theorem.   Modular units can be explicitly constructed by taking products and quotients of Siegel functions; these have many nice properties and have an explicit $q$-expansion at each cusp.   
  One downside to dealing with modular units is that they can have poles of high order and their $q$-expansion can have relatively large coefficients.   For example in the proof of Bilu and Parent, modular units on $X_G$ arise for which the order of the poles might not be uniformly bounded by a polynomial in $N$.\\

Let $\Delta$ be the modular discriminant function; it is a cusp form of weight $12$ for $\SL_2(\ZZ)$ that is everywhere nonzero when viewed as a function of the complex upper half-plane.    For a function $\varphi \in K(X_G)$ whose poles only occur at cusps,  $\varphi\cdot \Delta^m$ will be a modular form of weight $12m$ on $\Gamma(N)$ for all sufficiently large $m$.   

Now fix a positive integer $m$.  Consider the finite dimensional $K_G$-vector space  $M_{12m,G}$ consisting of those modular forms $f$ of weight $12m$ on $\Gamma(N)$ for which $f/\Delta^m$ lies in $K_G(X_G)$.    For each $f\in M_{12m,G}$, the function $f/\Delta^m\in K_G(X_G)$ has no poles away from the cusps.  One advantage of such functions is that they have a uniformly bounded number of poles; in fact, they have fewer than $m N^3/2$ total poles when counted with multiplicity.   

Let $W_m$ be the $K_G$-subspace of $M_{12m,G}$ consisting of modular forms $f$ for which $f/\Delta^m\in K_G(X_G)$ is regular at each cusp $c\in \Sigma$.    For a fixed $f \in W_m$, $\varphi:=f/\Delta^m \in K_G(X_G)$ will have all its poles at cusps and will be regular at all $c\in \Sigma$.   Since we want $\varphi$ to be nonconstant, we will also want to choose $f$ so that  so that it does not lie in the $1$-dimensional subspace $K_G \Delta^m$.  By taking $m$ sufficiently large, we will see from the Riemann--Roch theorem that $\dim_{K_G} W_m \geq 2$ and hence we can find a suitable $f$.   This is the source of our functions $\varphi$ in this paper.

However, to understand the growth of $\varphi$ near the cusps of $X_G$, we need to have suitable explicit bounds on the coefficients of the $q$-expansion of $f$ at each cusp.  As a first step, we will find a basis of the vector space $M_{12m,G}$ for which the $q$-expansions at each cusp have coefficients in $\ZZ[\zeta_N]$ which can be explicitly bounded with respect to any absolute value of $\QQ(\zeta_N)$.   Our basis will be expressed in terms of Eisenstein series of weight $1$ on $\Gamma(N)$.  We then use a version of Siegel's lemma to show the existence of a modular form $f \in W_m- K_G\Delta^m$ with explicit bounds on the coefficients of the $q$-expansion at each cusp.\\

Let us give a brief overview of the sections of the paper.  In \S\ref{S:first modular curve}, we give background on modular curves.  In \S\ref{S:nearby}, we will show how around each cusp of $X_G$ we can express rational functions analytically with respect to different places $v$.  In \S\ref{S:desired varphi}, we state a theorem about the existence of a function $\varphi \in K(X_G)$ with the required properties.   Assuming the existence of such a function $\varphi$, we then prove Theorem~\ref{T:main} in \S\ref{S:main proof}.   The existence of the desired $\varphi$ is proved in \S\ref{S:nouveau proof} after several sections on modular forms.   

In \S\ref{S:modular forms}, we give some background on modular forms.  In particular, we define the vector spaces $M_{k,G}$ of modular forms in \S\ref{SS:MkG} and we give an explicit generating set for $M_{k,G}$ in terms of Eisenstein series in \S\ref{SS:KM}.  Using this generating set, we prove in \S\ref{S:basis} the existence of a basis of $M_{k,G}$ whose $q$-expansions at each cusp have integral coefficients that can be bounded explicitly.  In \S\ref{S:RR}, we define the subspace $W_m$ of $M_{12m,G}$ for each positive integer $m$ and use the Riemann--Roch to show that $\dim_{K_G} W_m \geq 2$ for $m$ large enough.   For $m$ large enough, we prove in \S\ref{S:existence certain} the existence of a modular form $f\in W_m - K_G\Delta^m$ so that its $q$-expansions at the cusps have integral coefficients that can be bounded explicitly.

\subsection{Notation} \label{SS:notation}

Let $\zeta_N$ be the primitive $N$-th root of unity $e^{2\pi i /N}$ in $\CC$.   For each positive integer $N$, we define $q_N:=e^{2\pi i \tau/N}$ which we view as a function in $\tau$ of the complex upper half-plane.  When used with $q$-expansions, we will often view $q_N$ as an indeterminate variable.  For a positive divisor $w$ of $N$, we have $q_N^{N/w}=q_w$.  We set $q:=q_1$.

For a field $K$, we define $\Gal_K:=\Gal(\Kbar/K)$, where $\Kbar$ is a fixed algebraic closure of $K$.  When $K\subseteq \CC$, we will always take $\Kbar$ to be the algebraic closure in $\CC$.

Consider a number field $L$.   We let $M_L$  be the set of places of $L$.  We let $M_{L,\infty}\subseteq M_L$ be the set of infinite places.  Take any place $v$ of $L$.  We let $L_v$ be the completion of $L$ with respect to $v$.   We let $|\cdot|_v$ be the corresponding absolute value on $L_v$ normalized so that $|2|_v=2$ if $v$ is infinite and $|p|_v=p^{-1}$ if $v$ induces the $p$-adic topology on $\QQ$.   For an algebraic closure $\Lbar_v$ of $L_v$, the absolute value $|\cdot|_v$ uniquely extends to it.  Define the integer $d_v=[L_v:\QQ_v]$, where $\QQ_v$ is the completion of $\QQ$ in $L_v$.    For any nonzero $a\in L$, we have the \emph{product formula} $\prod_{v\in M_L} |a|_v^{d_v}=1$.

Let $C$ be a smooth projective and geometrically irreducible curve defined over a field $K$.  For any point $c\in C(K)$, let $\ord_c \colon K(C)\twoheadrightarrow \ZZ \cup \{+\infty\}$ be the discrete valuation whose valuation ring consists of the rational functions that are regular at $c$.

\section{Background: modular curves} \label{S:first modular curve}

Fix a positive integer $N$ and a group $G \subseteq \GL_2(\ZZ/N\ZZ)$.   The goal of this section is to give a quick definition of the modular curve $X_G$.   While we could define $X_G$ as a coarse moduli space, we will instead define it by explicitly giving its function field.

Let $(\ZZ/N\ZZ)^\times \xrightarrow{\sim} \Gal(\QQ(\zeta_N)/\QQ)$, $d\mapsto\sigma_d$ be the group isomorphism for which $\sigma_d(\zeta_N)=\zeta_N^d$.  We define $K_G=\QQ(\zeta_N)^{\det(G)}$ to be the subfield of $\QQ(\zeta_N)$ fixed by $\sigma_d$ for all $d\in \det(G)$.

\subsection{Modular functions} \label{SS:modular functions}

The group $\SL_2(\ZZ)$ acts by linear fractional transformations on the complex upper half-plane $\calH$ and the extended upper half-plane $\calH^*=\calH\cup \QQ \cup \{\infty\}$.   Let $\Gamma$ be a congruence subgroup of $\SL_2(\ZZ)$.  The quotient $\calX_\Gamma:=\Gamma\backslash \calH^*$ is a smooth compact Riemann surface (away from the cusps and elliptic points, use the analytic structure coming from $\calH$ and extend to the full quotient).   Denote the field of meromorphic functions on $\calX_{\Gamma}$ by $\CC(\calX_{\Gamma})$.

Fix a positive integer $N$ and let $\Gamma(N)$ be the congruence subgroup of $\SL_2(\ZZ/N\ZZ)$ consisting of matrices that are congruent to the identity modulo $N$.  Every $f\in \CC(\calX_{\Gamma(N)})$ gives rise to a meromorphic function on $\calH$ that satisfies
\[
f(\tau) = \sum_{n\in \ZZ} a_n(f) q_N^{n}
\]
for all $\tau\in \calH$ with sufficiently large imaginary component, where $q_N:=e^{2\pi i \tau/N}$ and the $a_n(f)$ are unique complex numbers which are nonzero for only finitely many $n <0$.   This Laurent series in $q_N$ is called the \defi{$q$-expansion} of $f$ (at the cusp at infinity) and it determines $f$ uniquely.  Let $\calF_N$ be the subfield of $\CC(\calX_{\Gamma(N)})$ consisting of all meromorphic functions $f$ for which $a_n(f)$ lies in $\QQ(\zeta_N)$ for all $n\in \ZZ$.  For example, $\calF_1=\QQ(j)$, where $j$ is the modular $j$-invariant.  

We now describe a right action $*$ of $\GL_2(\ZZ/N\ZZ)$ on $\calF_N$.  Note that the following two lemmas are both consequences of Theorem~6.6 and Proposition~6.9 of \cite{MR1291394}. 

\begin{lemma} \label{L:Shimura basic}
There is a unique right action $*$ of $\GL_2(\ZZ/N\ZZ)$ on the field $\calF_N$ such that the following hold for all $f\in \calF_N$:
\begin{itemize}
\item
For $A\in \SL_2(\ZZ/N\ZZ)$, we have $(f*A)(\tau) = f(\gamma\tau)$, where $\gamma\in \SL_2(\ZZ)$ is any matrix congruent to $A$ modulo $N$.
\item \label{L:Shimura basics i}
For $A=\left(\begin{smallmatrix}1 & 0 \\0 & d\end{smallmatrix}\right) \in \GL_2(\ZZ/N\ZZ)$, the $q$-expansion of $f*A$ is $\sum_{n\in \ZZ} \sigma_d(a_n(f)) q_N^{n}$. 
\end{itemize}
\end{lemma}

For each subgroup $G$ of $\GL_2(\ZZ/N\ZZ)$, let $\calF_N^G$ be the subfield of $\calF_N$ fixed by $G$ under the action $*$ of Lemma~\ref{L:Shimura basic}.  

\begin{lemma} \label{L:Shimura basics 2}
\begin{romanenum}
\item \label{L:Shimura basics 2 i}
The matrix $-I$ acts trivially on $\calF_N$ and the right action of $\GL_2(\ZZ/N\ZZ)/\{\pm I\}$ on $\calF_N$ is faithful.
\item \label{L:Shimura basics 2 ii}
We have $\calF_N^{\GL_2(\ZZ/N\ZZ)}=\calF_1=\QQ(j)$ and $\calF_N^{\SL_2(\ZZ/N\ZZ)}=\QQ(\zeta_N)(j)$.
\item \label{L:Shimura basics 2 iii}
The field $\QQ(\zeta_N)$ is algebraically closed in $\calF_N$.
\end{romanenum}
\end{lemma}

\subsection{Modular curves} 
\label{SS:modular curve definition}

Take any subgroup $G$ of $\GL_2(\ZZ/N\ZZ)$.  From Lemma~\ref{L:Shimura basics 2}, we find that the field $\calF_N^G$ has transcendence degree $1$ and that $K_G$ is the algebraic closure of $\QQ$ in $\calF_N^G$.  We define the \defi{modular curve} $X_G$ to be the smooth, projective and geometrically irreducible curve over $K_G$ that has function field $\calF_N^G$.   We have $j\in \calF_N^G=K_G(X_G)$ which gives a nonconstant morphism 
\[
j\colon X_G \to \PP^1_{K_G}
\]
whose degree we denote by $\mu$.   

Let $\bbar{G}$ be the subgroup of $\GL_2(\ZZ/N\ZZ)$ generated by $G$ and $-I$.  Observe that $\calF_N^G=\calF_N^{\bbar{G}}$ and hence $X_G=X_{\bbar{G}}$.  We have $\mu=[\SL_2(\ZZ/N\ZZ): \bbar{G} \cap \SL_2(\ZZ/N\ZZ)]$.  

Let $\Gamma_G$ be the congruence subgroup of $\SL_2(\ZZ)$ consisting of those matrices whose image modulo $N$ lies in $\bbar{G}\cap \SL_2(\ZZ/N\ZZ)$; it contains $-I$.  We have an inclusion $\CC \cdot K_G(X_G) \subseteq \CC(\calX_{\Gamma_G})$ of fields that both have degree $\mu= [\SL_2(\ZZ): \Gamma_G]$ over $\CC(j)$.  Therefore, $\CC(\calX_{\Gamma_G})=\CC(X_G)$.   Using this equality of function fields, we shall identify $X_G(\CC)$ with the Riemann surface $\calX_{\Gamma_G}$.  

\subsection{Cusps}

Fix notation as in \S\ref{SS:modular curve definition}.  Let $\calC_G$ be the set of cusps of $\calX_{\Gamma_G}=X_G(\CC)$. i.e., the set of orbits of $\Gamma_G$ on $\QQ\cup\{\infty\}$ or equivalently the points above $\infty$ under the morphism $j$.   We have $\calC_G \subseteq X_G(\QQ(\zeta_N))$, see Lemma~\ref{L:cusps of X(N)}(\ref{L:cusps of X(N) ii}).

Let $c_\infty$ be the cusp at infinity, i.e., the orbit containing $\infty$.  We have a bijection 
 \[
 \Gamma_G\backslash \SL_2(\ZZ)/U \xrightarrow{\sim} \calC_G, \quad \Gamma_G A U \mapsto A\cdot c_\infty,
 \]   
 where $U$ is the group of upper triangular matrices in $\SL_2(\ZZ)$ generated by $-I$ and $\left(\begin{smallmatrix}1 & 1 \\0 & 1\end{smallmatrix}\right)$.   Since the level of $\Gamma_G$ divides $N$, we find that the cusp $A\cdot c_\infty$ of $X_G$ depends only on $A$ modulo $N$.   In particular, it makes sense to talk about the cusp $A\cdot c_\infty$ for any $A\in \SL_2(\ZZ/N\ZZ)$.

Now take any cusp $c\in \calC_G$.  We define $w_c$ to be the ramification index of $c$ over $\infty$ with respect to the morphism $j\colon X_G\to \PP^1_{K_G}$.  Equivalently, $w_c$ is the \defi{width} of the cusp $c$ for the congruence subgroup $\Gamma_G$.  The integer $w_c$ divides $N$.  Since $\mu$ is the degree of $j\colon X_G\to \PP^1_{K_G}$, we have
\begin{align} \label{E:sum of widths}
\sum_{c \in \calC_G} w_c = \mu.
\end{align}

Take any $A\in \SL_2(\ZZ/N\ZZ)$ for which $A\cdot c_\infty =c$.   Then for any $f\in \CC(X_G)$, the $q$-expansion of $f*A \in \calF_N$ is a Laurent series in the variable $q_{w_c}$.   

\subsection{The modular curve $X(N)$} \label{SS:XN}

Let $X(N)$ be the modular curve corresponding to the subgroup of $\GL_2(\ZZ/N\ZZ)$ consisting of matrices of the form $\left(\begin{smallmatrix}1 & 0 \\0 & *\end{smallmatrix}\right)$; it is a curve defined over $\QQ$.   The group $\SL_2(\ZZ/N\ZZ)$ has a left action on the curve $X(N)_{\QQ(\zeta_N)}$ corresponding to the right action on the function field $\QQ(\zeta_N)(X(N))=\calF_N$ given in Lemma~\ref{L:Shimura basic}.    

\begin{lemma} \label{L:cusps of X(N)}
\begin{romanenum}
\item \label{L:cusps of X(N) i}
The cusp $c_\infty$ of $X(N)$ is defined over $\QQ$.   
\item \label{L:cusps of X(N) ii}
For any subgroup $G$ of $\GL_2(\ZZ/N\ZZ)$, we have $\calC_G \subseteq X_G(\QQ(\zeta_N))$.  The set $\calC_G$ is stable under the action of $\Gal(\QQ(\zeta_N)/K_G)$.  
\end{romanenum}
\end{lemma}
\begin{proof}
Let $\calC$ be the set of cusps of in $X(N)(\CC)$.  From \cite[Lemma 5.2]{OpenImage}, we find that $\calC\subseteq X(N)(\QQ(\zeta_N))$ and that $c_\infty$ is defined over $\QQ$.   Take any subgroup $G\subseteq \GL_2(\ZZ/N\ZZ)$ and let $\pi\colon X(N)_{\QQ(\zeta_N)} \to (X_G)_{\QQ(\zeta_N)}$ be the morphism corresponding to the inclusion $\QQ(\zeta_N)(X_G) \subseteq \calF_N = \QQ(\zeta_N)(X(N))$ of function fields.   We have $\pi(\calC)=\calC_G$.   Therefore, $\calC_G \subseteq \pi(X(N)(\QQ(\zeta_N))\subseteq X_G(\QQ(\zeta_N))$, where the last equality uses that $\pi$ is defined over $\QQ(\zeta_N)$. Since $j\colon X_G \to \PP^1_{K_G}$ is defined over $K_G$, the set $j^{-1}(\infty)=\calC_G\subseteq X_G(\QQ(\zeta_N))$ is stable under the action of $\Gal(\QQ(\zeta_N)/K_G)$. 
\end{proof}

\begin{remark} \label{SS:X(N) moduli}
Suppose that $N>2$.  The modular curve $X(N)_{\QQ(\zeta_N)}$ has an alternate description as a fine moduli space.  We refer to Deligne and Rapoport \cite{MR337993} where this theory is fully developed. They define a smooth projective curve $M_N$ over $\QQ(\zeta_N)$ that is the moduli space for generalized elliptic curves with a  level $N$ structure (we use $N>2$ to ensure the moduli problem is represented by a scheme and we have also base extended to $\QQ(\zeta_N)$ from the ring $\ZZ[1/N,\zeta_N]$).   The curve $M_N$ has $\phi(N)$ irreducible components.  Our curve $X(N)_{\QQ(\zeta_N)}$ can be identified with an irreducible $C$ component of $M_N$ (in particular, the irreducible component consisting of generalized elliptic curves with a level $N$ structure so that the Weil pairing of the basis is $\zeta_N$).  To prove this identification one need only identify the function field of $C$ with $\calF_N = \QQ(\zeta_N)(X(N))$; thus can be done using the material in \cite[Chapter VII, \S4]{MR337993} which compares modular forms with the classical theory.
\end{remark}

For every $f\in \QQ(X(N))$, its $q$-expansion (at $c_\infty$) lies in $\QQ(\!(q_N)\!)$ and it lies in $\QQ[\![q_N]\!]$ when $f$ is regular at $c_\infty$.   So with $\OO$ the local ring of rational functions on $X(N)$ that are regular at $c_\infty$, $q$-expansions induces an isomorphism between the completion of $\OO$ with $\QQ[\![q_N]\!]$. Equivalently, $q$-expansions induces an isomorphism between the formal completion of the curve $X(N)$ along $c_\infty$ with the formal spectrum of $\QQ[\![q_N]\!]$.

Now consider \emph{any} field $F\supseteq \QQ$.  If $\OO$ is the local ring of rational functions on $X(N)_F$ that are regular at $c_\infty$, then we obtain an isomorphism between the completion of $\OO$ with $F[\![q_N]\!]$.     So for any $f \in F(X(N)_F)=F(X(N))$, we obtain a $q$-expansion in $F(\!(q_N)\!)$ that lies in $F[\![q_N]\!]$ when $f$ is regular at $c_\infty$.

Consider any field $F\supseteq \QQ(\zeta_N)$.  The group $\SL_2(\ZZ/N\ZZ)$ acts on $X(N)_F$ and hence acts on $F(X(N))$.  The field $F(X(N))$ is the compositum of $F$ and $\calF_N$.  The group $\SL_2(\ZZ/N\ZZ)$ acts on the field $F(X(N))$ as the identity on $F$ and as our action $*$ on $\calF_N$; we also denote this action by $*$.

\section{Analytic expansions at the cusps} \label{S:nearby}

Fix an integer $N>2$.   Fix a number field $L\supseteq \QQ(\zeta_N)$ and a place $v$ of $L$.  

Take any cusp $c$ of $X(N)$ and choose an $A\in \SL_2(\ZZ/N\ZZ)$ for which $A\cdot c_\infty=c$.  Consider a rational function $f $ in $\Lbar_v(X(N))$.  As noted in \S\ref{SS:XN}, $f*A$ has a $q$-expansion in $\Lbar_v(\!(q_N)\!)$ that we denote by $\sum_{n\in \ZZ} a_n(f*A) q_N^n$.   We will show that for all points in $X(N)(\Lbar_v)$ near $c$, but maybe not equal to $c$, the function $f$ can be expressed analytically in terms of the Laurent series of the the $q$-expansion of $f*A$.   In order to make this precise, we will define a subset $\Omega_{c,v} \subseteq X(N)(\Lbar_v)$ that only contains the one cusp $c$ and whose interior is an open neighborhood of $c$.
 
For a subgroup $G$ of $\GL_2(\ZZ/N\ZZ)$, we will also define similar subsets $\Omega_{c,v}$ of $X_G(\Lbar_v)$ in \S\ref{SS:nearby 3}.   

\subsection{The modular curve $X(N)$} \label{SS:XN specializations}

As noted in \S\ref{SS:XN}, the group $\SL_2(\ZZ/N\ZZ)$ acts on $X(N)_{F}$ and  $F(X(N))$ for any field $F\supseteq \QQ(\zeta_N)$.   For each $A\in \SL_2(\ZZ/N\ZZ)$, let $\iota_A$ be the corresponding automorphism of $X(N)_{\QQ(\zeta_N)}$.    For any place $v$ of $L$, $\iota_A$ gives a homeomorphism $X(N)(\Lbar_v)\xrightarrow{\sim} X(N)(\Lbar_v)$.   Define the open ball $B_v:=\{a\in \Lbar_v: |a|_v<1\}$ of $\Lbar_v$.    

\begin{prop} \label{P:specializations}
Fix a place $v$ of $L$.  Take a matrix $A\in \SL_2(\ZZ/N\ZZ)$ and define the cusp $c:=A\cdot c_\infty$ of $X(N)$.   Then there is a unique continuous map 
\[
\psi_{A,v}\colon B_v \to X(N)(\Lbar_v)
\]
such that the following hold:
\begin{alphenum}
\item \label{P:specializations a}
Take any rational function $f\in \Lbar_v(X(N))$ and let $r$ be the radius of convergence of the power series $\sum_{n=0}^\infty a_n(f*A) x^n$ in $\Lbar_v[\![x]\!]$.   Then for all nonzero $t\in B_v$ with $|t|_v<r$, we have
\[
f(\psi_{A,v}(t))=\sum_{n\in \ZZ} a_n(f*A) t^n
\]
in $\Lbar_v$.
\item \label{P:specializations b}
We have $\psi_{A,v}(0)=c$.  For any nonzero $t\in B_v$, $\psi_{A,v}(t)$ is not a cusp and 
\[
j(\psi_{A,v}(t))= J(t^N)
\]
in $\Lbar_v$, where $J(q)=q^{-1}+744+196884q+21493760q^2 + \cdots \in \ZZ(\!(q)\!)$ is the $q$-expansion of $j$.  When $v$ is finite, we have $|j(\psi_{A,v}(t))|_v = |t|_v^{-N}>1$ for all nonzero $t\in B_v$.
\end{alphenum}
We have $\psi_{A,v}=\iota_A \circ \psi_{I,v}$.  
\end{prop}
\begin{proof}

We first assume that there is a unique continuous map $\psi_{I,v}\colon B_v\to X(N)(\Lbar_v)$ satisfying (\ref{P:specializations a}) and (\ref{P:specializations b}).  Define the continuous map $\psi_{A,v} := \iota_A \circ \psi_{I,v}\colon B_v \to  X(N)(\Lbar_v)$.   We have $\psi_{A,v}(0)=\iota_A(\psi_{I,v}(0))=\iota_A(c_\infty)=A\cdot c_\infty=c$.   Take any $f\in \Lbar_v(X(N))$.   We have $f*A \in \Lbar_v(X(N))$.  Let $r$ be the radius of convergence of $\sum_{n=0}^\infty a_n(f*A) x^n \in \Lbar_v[\![x]\!]$.  For any nonzero $t\in B_v$ with $|t|_v<r$, we have 
\[
f(\psi_{A,v}(t))=f(\iota_A(\psi_{I,v}(t))=(f*A)(\psi_{I,v}(t))=\sum_{n\in \ZZ} a_n(f*A) t^n \in \Lbar_v,
\] 
where the last equality uses our assumption that $\psi_{I,v}$ exists with the expected properties.   For any nonzero $t\in B_v$, we also have $j(\psi_{A,v}(t))=j(\iota_A(\psi_{I,v}(t)))=(j*A)(\psi_{I,v}(t))=j(\psi_{I,v}(t))$.   So (\ref{P:specializations b}) holds for $\psi_{A,v}$ by the corresponding property that we have assumed holds for $\psi_{I,v}$.   Therefore, $\psi_{A,v}$ satisfies (\ref{P:specializations a}) and (\ref{P:specializations b}).

We now show that $\psi_{A,v}$ is unique.  A similar argument as above shows that for any continuous $\psi_{A,v}$ satisfying (\ref{P:specializations a}) and (\ref{P:specializations b}), the map $\iota_A^{-1} \circ \psi_{A,v} \colon B_v \to X(N)(\Lbar_v)$ is continuous and satisfies the same properties (\ref{P:specializations a}) and (\ref{P:specializations b}) as $\psi_{I,v}$.  By our assumption that $\psi_{I,v}$ is unique, we deduce that  $\psi_{I,v}=\iota_A^{-1} \circ \psi_{A,v}$.  This proves the uniqueness of $\psi_{A,v}$ and that $\psi_{A,v}=\iota_A \circ \psi_{I,v}$.  

So without loss of generality, we may assume that $A=I$ and hence $c=c_\infty$.\\

Before starting the case where $v$ is infinite, let us recall the classic situation over $\CC$. Define $B:=\{t \in \CC: |t|<1\}$.   Let 
\[
\pi\colon \calH \cup\{\infty\} \to \Gamma(N)\backslash\calH^*=\calX_{\Gamma(N)} = X(N)(\CC)
\]
be the natural quotient map with the last equality being the identification from \S\ref{SS:modular curve definition}.   The map $\pi$ can be expressed as the composition of $\calH\cup \{\infty\} \twoheadrightarrow B$, $\tau\mapsto e^{2\pi i\tau/N}$ (where $\infty\mapsto 0$) with a unique function $\psi\colon B \to X(N)(\CC)$.    The map $\psi$ is continuous and its image is all of $X(N)(\CC)$ except for those cusps that are not $c_\infty$.   We have $\psi(0)=c_\infty$ and $\psi(t)$ is not a cusp for all nonzero $t\in B$.  Take any $f\in \CC(X(N))=\CC(\calX_{\Gamma(N)})$ that is regular at $c_\infty$.   Let $\sum_{n\in \ZZ} a_n(f) q_N^n \in \CC(\!(q_N)\!)$ be the $q$-expansion of $f$ and let $r$ be its radius of convergence of $\sum_{n=0}^\infty a_n(f) x^n \in \CC[\![x]\!]$.   Take any nonzero $t\in B$ with $|t|<r$ and choose a $\tau \in \calH$ for which $t=e^{2\pi i \tau/N}$.    We have \[
f(\psi(t))=(f\circ \pi)(\tau) = \sum_{n\in \ZZ} a_n(f) e^{2\pi i \tau n/N} = \sum_{n\in \ZZ} a_n(f) t^n.
\]   
In the special case $f=j$, we have a $q$-expansion $J(q) =q^{-1}+744+ \cdots \in \ZZ[\![q]\!]$ and hence $j(\psi(t)) = J(t^N)$ for all nonzero $t\in B_v$.

We now consider the case where $v$ is infinite.  Fix an embedding $\sigma\colon L \hookrightarrow \CC$ that satisfies $|a|_v=|\sigma(a)|$ for all $a\in L$.  By continuity, $\sigma$ extends uniquely to an isomorphism $\bbar\sigma\colon \Lbar_v=L_v \xrightarrow{\sim} \CC$ of fields that respects absolute values (the field $L_v$ is not real since it contains $\zeta_N$ with $N>2$).  Since $X(N)$ and the cusp $c_\infty$ are defined over $\QQ$, $\bbar\sigma$ induces a homeomorphism $\bbar\sigma_*\colon X(N)(\Lbar_v) \xrightarrow{\sim} X(N)(\CC)$ that maps $c_\infty$ to itself.    Define the continuous function
\[
\psi_{I,v}\colon B_v\xrightarrow{\sim} B \xrightarrow{\psi} X(N)(\CC) \xrightarrow{\sim} X(N)(\Lbar_v),
\]
where the first map is given by $\bbar\sigma$ and the third map is the inverse of $\bbar\sigma_*$.  The image of $\psi_{I,v}$ is equal to $X(N)(\Lbar_v)$ with all the cusps except $c_\infty$ removed since $\psi$ has this property.   The desired properties for $\psi_{I,v}$ are now immediate consequences of the analogous properties of $\psi$.  \\

  The process of $q$-expansions induces an isomorphism between the formal completion of the curve $X(N)_{\QQ(\zeta_N)}$ along $c_\infty$ with the formal spectrum of $\QQ(\zeta_N)[\![q_N]\!]$.   We will require a stronger version of Deligne and Rapoport that we now recall.   Define the ring $R:=(\ZZ[\zeta_N][\![ q_N ]\!])\otimes_{\ZZ[\zeta_N]} \QQ(\zeta_N)$; we can view it as a subring of $\QQ(\zeta_N)[\![q_N]\!]$.   From Deligne and Rapoport \cite[Chapter VII, Corollary 2.4]{MR337993} with Remark~\ref{SS:X(N) moduli}, the Tate curve produces a morphism
\begin{align}\label{E:DR}
\Spec R \to X(N)_{\QQ(\zeta_N)}
\end{align}
that induces an isomorphism between the formal completion of $X(N)_{\QQ(\zeta_N)}$ along $c_\infty$ and the formal spectrum of $\QQ(\zeta_N)[\![q_N]\!]$.  

We now consider a finite place $v$ of $L$.  Define the $\Lbar_v$-algebra $R':=R \otimes_{\QQ(\zeta_N)} \Lbar_v$.  Take any $t\in B_v$.   Evaluating the power series in $R' \subseteq \Lbar_v[\![q_n]\!]$ at $t$ gives a  homomorphism $\phi_t\colon R'\to \Lbar_v$ of $\Lbar_v$-algebras.  Composing $\phi_t^*\colon \Spec \Lbar_v \to \Spec R'$ with the morphism $\Spec R' \to X(N)_{\Lbar_v}$ obtained from base changing (\ref{E:DR}) produces a point $\psi_{I,v}(t)$ in $
X(N)(\Lbar_v)$.    We thus have defined a map 
\[
\psi_{I,v} \colon B_v \to X(N)(\Lbar_v)
\]
and it is continuous.  Note that $\psi_{I,v}(0)=c$.

Take any $t\in B_v$ and let $R_t$ be the subring of $\Lbar_v[\![q_N]\!]$ consisting of those series whose radius of convergence is strictly greater than $|t|_v$.  By base changing (\ref{E:DR}) by $\Lbar_v$ and using the inclusion $R' \subseteq R_t$, we obtain a morphism $\Spec R_t \to X(N)_{\Lbar_v}$ that induces an isomorphism between the formal completion of $X(N)_{\Lbar_v}$ at $c_\infty$ and the formal spectrum of the power series ring $\Lbar_v[\![q_N]\!]$.  Evaluating power series in $R_t$ at $t$ induces a morphism $\Spec \Lbar_v \to \Spec R_t$ which after composing with $\Spec R_t \to X(N)_{\Lbar_v}$ gives the $\Lbar_v$-point $\psi_{I,v}(t)$. So for any rational function $f\in \Lbar_v(X(N))$ regular at $c_\infty$ for which $\sum_{n=0}^\infty a_n(f) q_N^n \in \Lbar_v[\![x]\!]$ lies in $R_t$, we have $f(\psi_{I,v}(t))=  \sum_{n=0}^\infty a_n(f) t^n$.

The function $j^{-1}$ is regular at $c_\infty$ and its $q$-expansion is $h(q):=J(q)^{-1}=q - 744q^2 +\cdots \in \ZZ[\![q]\!]\subseteq R$.   Take any nonzero $t\in B_v$.  We have $j^{-1}(\psi_{I,v}(t))=t^{N} - 744t^{2N} +\cdots=h(t^N)$ in $\Lbar_v$.   Since $|t|_v<1$, this series implies that $|j^{-1}(\psi_{I,v}(t))|_v=|t|_v^{N}$.   In particular, $j^{-1}(\psi_{I,v}(t))$ is nonzero since $t$ is nonzero.   Therefore, $\psi_{I,v}(t)$ is not a cusp and $j(\psi_{I,v}(t))=h(t^N)^{-1}=J(t^N)$.   We have also shown that $| j(\psi_{I,v}(t)) |_v = |t|_v^{-N}>1$.

We have now shown that $\psi_{I,v}\colon B_v \to X(N)(\Lbar_v)$ satisfies (\ref{P:specializations b}).  It also satisfies (\ref{P:specializations a}) since we have shown that (\ref{P:specializations a}) holds for $f \in \Lbar_v(x(N))$ regular at $c_\infty$ and for $j$ (for any $f\in \Lbar_v(X(N))$, there is an integer $m$ for which $f \cdot j^m$ is regular at $c$).

Finally, it remains to prove the uniqueness of $\psi_{I,v}$; we have already proved the existence.   For $\tau\in \calH$, let $\wp(z;\tau)$ be the Weierstrass elliptic function for the lattice $\ZZ \tau + \ZZ \subseteq \CC$.   For integers $0\leq r , s < N$ that are not both zero, define the function $x_{(r,s)}(\tau):=36 E_4(\tau) E_6(\tau) \Delta(\tau)^{-1} \cdot (2\pi i)^{-2} \wp(\tfrac{r}{N}\cdot \tau + \tfrac{s}{N}; \tau)$, where $E_4$ and $E_6$ are the usual Eisenstein series of weight $4$ and $6$, respectively, on $\SL_2(\ZZ)$.   The function $x_{(r,s)}$ lies in $\calF_N$ with cusps only at the poles, see~\cite[Lemma~6.1]{OpenImage}.  The $q$-expansion of $12\cdot (2\pi i)^{-2} \wp(\tfrac{r}{N}\cdot \tau + \tfrac{s}{N}; \tau)$ is of the form $\sum_{n=0}^\infty a_n q_N^n$ with $a_n\in \ZZ[\zeta_N]$ and $|a_n|_v \ll n$ for $n\geq 1$, see the proof of \cite[Lemma~6.1]{OpenImage} for the explicit $q$-expansion.   From this, we deduce that the $q$-expansion of each $x_{(r,s)}$ will have coefficients in $\QQ(\zeta_N)\subseteq L$ and will have radius of convergence at least $1$ when viewed as having coefficients in $\Lbar_v$.  In particular, the value $x_{(r,s)}(\psi_{I,v}(t))$ depends only on $t$ by (\ref{P:specializations a}) (and not the choice of $\psi_{I,v}$).   By (\ref{P:specializations b}), $j(\psi_{I,v}(t))$ also depends only on $t$.  However, the functions $x_{(r,s)}$ along with $j$ generate the field $\calF_N$, cf.~\cite[Lemma~6.1]{OpenImage}, and hence also generate $\Lbar_v(X(N))$ over $\Lbar_v$.   This implies that $\psi_{I,v}(t)$ does not depend on the choice of $\psi_{I,v}$.
\end{proof}

Let $\calC$ be the set of cusps of $X(N)$; it is a subset of $X(N)(L)$ by Lemma~\ref{L:cusps of X(N)}(\ref{L:cusps of X(N) ii}).
Take any cusp $c\in \calC$ and any place $v$ of $L$.  Choose an $A\in \SL_2(\ZZ/N\ZZ)$ for which $A\cdot c_\infty =c$.  With $\psi_{A,v}$ as in Proposition~\ref{P:specializations}, define the following subset of $X(N)(\Lbar_v)$:
\[
\Omega_{c,v} := \begin{cases}  \psi_{A,v}(B_v) & \text{if $v$ is finite,}\\
\psi_{A,v}(\{t\in \Lbar_v: |t|_v\leq e^{-\pi\sqrt{3}/N}\}) & \text{if $v$ is infinite.} 
\end{cases}
\]
The only cusp in the set $\Omega_{c,v}$ is $c$.

\begin{lemma}  \label{L:Omega first}
Take any place $v$ of $L$.
\begin{romanenum}
\item \label{L:Omega first i}
For each cusp $c\in \calC$, $\Omega_{c,v}$ does not depend on the choice of $A\in \SL_2(\ZZ/N\ZZ)$ for which $A\cdot c_\infty=c$.
\item \label{L:Omega first ii}
If $v$ is finite, then $\bigcup_{c\in \calC} \Omega_{c,v} = \{P \in X(N)(\Lbar_v): P \text{ a cusp or } |j(P)|_v>1\}$.
\item \label{L:Omega first iii}
If $v$ is infinite, then $\bigcup_{c\in \calC} \Omega_{c,v} = X(N)(\Lbar_v)$.
\end{romanenum}
\end{lemma}
\begin{proof}
Consider any integer $b$.  We have $e^{2\pi i (\tau+b)/N}=\zeta_N^b q_N$.   So for a matrix $U= \pm \left(\begin{smallmatrix}1 & b \\0 & 1\end{smallmatrix}\right) \in \SL_2(\ZZ/N\ZZ)$ and a function $f\in \calF_N$, the $q$-expansion of $f*U$ is $\sum_{n\in \ZZ} a_n(f*U) q_N^n = \sum_{n\in \ZZ} a_n(f) \zeta_N^{bn} q_N^n$.  The same thing thus holds for the $q$-expansion of $f*U$ for any $f\in \Lbar_v(X(N))$.      

Take any matrices $A,A' \in \SL_2(\ZZ/N\ZZ)$ for which $A\cdot c_\infty=A'\cdot c_\infty$.    We have $A=A'\cdot U$ for some matrix $U$ of the form $\pm \left(\begin{smallmatrix}1 & b \\0 & 1\end{smallmatrix}\right)$.    For any $f\in \Lbar_v(X(N))$, the $q$-expansion of $f*A$ is $\sum_{n\in \ZZ} a_n(f*A) q_N^n = \sum_{n\in \ZZ} a_n((f*A')*U) q_N^n=\sum_{n\in \ZZ} a_n(f*A') \zeta_N^{bn} q_N^n$.   From the uniqueness of $\psi_{A,v}$ in Proposition~\ref{P:specializations}, we deduce that $\psi_{A,v}(t)=\psi_{A',v}(\zeta_N^b t)$ for all $t\in B_v$.   Part (\ref{L:Omega first i}) follows immediately since $|\zeta_N^b|_v=1$.  

Since $\psi_{A,v}=\iota_A\circ \psi_{I,v}$ by Proposition~\ref{P:specializations}, we find that $\calS:=\bigcup_{c\in \calC} \Omega_{c,v}$ is equal to the set \[\bigcup_{A\in \SL_2(\ZZ/N\ZZ)} \iota_A(\Omega_{c_\infty,v}).\]  In particular, $\calS \subseteq X(N)(\Lbar_v)$ is stable under the action of $\SL_2(\ZZ/N\ZZ)$.   Since $\SL_2(\ZZ/N\ZZ)$ acts transitively on the fibers of $j\colon X(N)(\Lbar_v)\to \PP^1(\Lbar_v)$, we have
\begin{align} \label{E:calS j}
\calS=\{ P \in X(N)(\Lbar_v) : j(P)\in j(\calS)\}.
\end{align}
Since $\calS$ contains all the cusps $\calC$, we need only compute the set $j(\calS-\calC) \subseteq \Lbar_v$ to prove (\ref{L:Omega first ii}) and (\ref{L:Omega first iii}).

First suppose that $v$ is finite.  We have $j(\calS-\calC) \subseteq \{a\in \Lbar_v: |a|_v>1\}$ by Proposition~\ref{P:specializations}(\ref{P:specializations b}).   Now take any $a\in \Lbar_v$ with $|a|_v>1$.  By \cite[Chapter V Lemma~5.1]{MR1312368}, there is a nonzero $t\in \Lbar_v$ with $|t|_v<1$ such that $a=J(t^N)$ with $J$ as in Proposition~\ref{P:specializations}(\ref{P:specializations b}).   By Proposition~\ref{P:specializations}(\ref{P:specializations b}), we have $a=J(t^N)=j(\psi_{I,v}(t)) \in j(\calS-\calC)$.   Therefore, $j(\calS-\calC) = \{a\in \Lbar_v: |a|_v>1\}$ and hence we obtain (\ref{L:Omega first ii}) from (\ref{E:calS j}).  

Finally suppose that $v$ is infinite.    Using (\ref{E:calS j}), we need only verify that $j(\calS-\calC) =\Lbar_v$ to prove part (\ref{L:Omega first iii}).    We have
\[
j(\calS-\calC)= \bigcup_{c\in \calC} j(\Omega_{c,v}-\{c\}) = \{J(t^N): t \in \Lbar_v-\{0\}, |t|_v\leq e^{-\pi\sqrt{3}/N}\},
\]
where the last equality uses the definition of $\Omega_{c,v}$ and Proposition~\ref{P:specializations}(\ref{P:specializations b}).  So after choosing an isomorphism $\Lbar_v \cong \CC$ that respects absolute values, we need only show that every $a\in \CC$ is of the form $J(t^N)$ for some nonzero $t\in \CC$ with $|t|\leq  e^{-\pi\sqrt{3}/N}$.  Take any $a\in \CC$.  There is a $\tau \in \calH$ for which $j(\tau)=a$, where we are viewing $j$ as a holomorphic function of the upper half-plane.   Using the action of $\SL_2(\ZZ)$ on $\calH$, we may further assume that $|\tau| \geq 1$ and $-1/2\leq \operatorname{Re}(\tau)\leq 1/2$.  In particular, $\operatorname{Im}(\tau) \geq \sqrt{3}/2$.   Define $u=e^{2\pi i \tau} \in \CC$.  We have $0<|u|\leq e^{-\pi \sqrt{3}}$.  Choose a $t\in \CC$ for which $u=t^N$.  We have $0<|t| \leq e^{-\pi \sqrt{3}/N}$ and $J(t^N)=J(u)=j(\tau)=a$.
\end{proof}

\begin{remark}
Consider the cusp $c:=c_\infty$.  The approach in \S3 of \cite{BiluParent} is to view $t_c:=q_N$ as a parameter on $X(N)$ that defines a $v$-analytic function on a neighborhood of $c$ in $X(N)(\Lbar_v)$ that maps $c$ to $0$.  They describe an open neighborhood of $c$ for which their function is defined and analytic; locally it is a homeomorphism whose inverse will agree with the restriction of our function $\psi_{I,v}$.   
Bilu and Parent take a similar approach for any cusp of a general modular curve $X_G$.   In the setting of this section, our set $\Omega_{c,v}$ will agree with that in \S3 of \cite{BiluParent} when $v$ is finite  (our sets are larger when $v$ is infinite). 
\end{remark}

\subsection{The sets $\Omega_{c,v}$ for a general modular curve}
\label{SS:nearby 3}

Take any subgroup $G$ of $\GL_2(\ZZ/N\ZZ)$.   The modular curve $X_G$ is defined over $K_G \subseteq \QQ(\zeta_N) \subseteq L$.   Let $\pi\colon X(N)_{\QQ(\zeta_N)} \to (X_G)_{\QQ(\zeta_N)}$ be the natural morphism corresponding to the inclusion of function fields $\QQ(\zeta_N)(X_G) \subseteq \calF_N$.  Let $\calC$ be the set of cusps in $X(N)(\QQ(\zeta_N))$.

Fix a place $v$ of $L$.  For each cusp $c\in \calC_G$, we define the subset
\[
\Omega_{c,v} := \pi\Big(\bigcup_{c' \in \calC\,\pi(c')=c} \Omega_{c',v}\Big)
\]
of $X_G(\Lbar_v)$, where the sets $\Omega_{c',v}\subseteq X(N)(\Lbar_v)$ are from \S\ref{SS:XN specializations}.   The set $\Omega_{c,v}$ contains $c$ and no other cusps.

\begin{lemma}  \label{L:union of neighborhoods}
For any place $v$ of $L$, we have
\[
\bigcup_{c \in \calC_G} \Omega_{c,v}=
\begin{cases}
\{P \in X_G(\Lbar_v): P \text{ is a cusp or }|j(P)|_v > 1\} &\text{ if $v$ is finite,}\\
X_G(\Lbar_v) &\text{ if $v$ is infinite.}
\end{cases}
\]
\end{lemma}
\begin{proof}
We have $\bigcup_{c \in \calC_G} \Omega_{c,v}=\pi(\bigcup_{c'\in \calC} \Omega_{c',v})$.   The lemma now follows from Lemma~\ref{L:Omega first}.
\end{proof}

\section{Properties of $\varphi$} \label{S:desired varphi}

Fix an integer $N>2$ and let $G$ be a subgroup of $\GL_2(\ZZ/N\ZZ)$ containing $-I$.  Let $\mu$ be the degree of the morphism $j\colon X_G\to \PP^1_{K_G}$.  The group $\Gal_{K_G}$ acts on the set of cusps $\calC_G$ of $X_G$ by Lemma~\ref{L:cusps of X(N)}(\ref{L:cusps of X(N) ii}).

Let $\Sigma$ be a proper subset of $\calC_G$ that is stable under the $\Gal_{K_G}$-action.  
Let $m$ be the smallest positive integer for which $m\sum_{c\in \calC_G-\Sigma} w_c > g$, where $g$ is the genus of $X_G$.   Fix a number field $L\supseteq \QQ(\zeta_N)$.   Define the real numbers
\[
\beta:=2 (2^3 4.5^{36m}  N^{108m+15} m^{72m+1})^{m\mu+1}   4.5^{12m} N^{36m+4},
\]
$C:=96.6 \cdot 0.1^{24m} N^{90m+4} $ and $C':=22.16  N^{144m+7} 0.024^{24m}$.  

The following theorem gives the existence of a rational function $\varphi \in K_G(X_G)$ suitable for our application of Runge's method; it makes use of the sets $\Omega_{c,v} \subseteq X_G(\Lbar_v)$ from \S\ref{SS:nearby 3}.   The proof will given in \S\ref{S:nouveau proof} after several sections discussing modular forms.    

\begin{thm} \label{T:nouveau varphi}
Fix notation and assumptions as above. There is a nonconstant function $\varphi \in K_G(X_G)$ that satisfies the following properties:  
\begin{alphenum}
\item \label{T:nouveau varphi A}
The function $\varphi$ is integral over $\ZZ[j]$, i.e., $\varphi$ is the root of a monic polynomial with coefficients in $\ZZ[j]$.
\item \label{T:nouveau varphi B}
The function $\varphi$ has no poles away from the cusps of $X_G$.

\item \label{T:nouveau varphi C}
Take any cusp $c\in\Sigma$.  The function $\varphi$ is regular at $c$.  Moreover, $\varphi(c)$ lies in $\ZZ[\zeta_N]$ and satisfies $|\varphi(c)|_v \leq \beta m^{24m}$ for all infinite places $v$ of $L$.

\item \label{T:nouveau varphi D}
For any cusp $c\in \Sigma$, place $v$ of $L$ and point $P\in \Omega_{c,v}-\{c\}$, we have 
\[
|\varphi(P)-\varphi(c)|_v \leq
\begin{cases}
|j(P)|_v^{-1/w_c} & \text{ if $v$ is finite,}\\
|j(P)|_v^{-1/w_c}\cdot  \beta C & \text{ if $v$ is infinite and $|j(P)|_v >3500$,}\\
\beta C/2 & \text{ if $v$ is infinite.}
\end{cases}
\]

\item \label{T:nouveau varphi E}
Let $K$ be any number field with $K_G\subseteq K \subseteq L$ and let $\Sigma'$ be a $\Gal_{K}$-orbit of $\Sigma$.  Let $w$ be the integer for which $w_c=w$ for all $c\in \Sigma'$.  Then there is a nonzero $\xi\in \OO_K$ such that the following hold:
\begin{itemize}
\item
We have $|\xi|_v \leq (\beta C')^{w|\Sigma'|}$ for all infinite place $v$ of $L$.  
\item
Consider any point $P\in Y_G(K)$ for which $\varphi(P)=\varphi(c)$ and $P\in \Omega_{c,v}$   for some cusp $c\in \Sigma'$ and place $v$ of $L$.  Then
\[
|\xi|_v \leq
\begin{cases}
|j(P)|_v^{-w} & \text{ if $v$ is finite,}\\
|j(P)|_v^{-w} \cdot (\beta C')^{w|\Sigma'|} & \text{ if $v$ is infinite and $|j(P)|_v>3500$.}
\end{cases}
\]
\end{itemize}
\end{alphenum}
\end{thm}

For later use, we now gives some simple bounds on $m$ and $\mu$ in terms of $N$.

\begin{lemma} \label{L:trivial bounds}
Fix notation as above.  
\begin{romanenum}
\item \label{L:trivial bounds i}
We have $\mu \leq \tfrac{1}{2} N^3$, $\mu+1 \leq \tfrac{29}{54} N^3$ and $\mu+2 \leq \tfrac{31}{54} N^3$.
\item  \label{L:trivial bounds ii}
We have $m \leq \tfrac{1}{24} N^3$. 
\end{romanenum}
\end{lemma}
\begin{proof}
Since $G$ contains $-I$, we have $\mu=[\SL_2(\ZZ/N\ZZ):\SL_2(\ZZ/N\ZZ) \cap G]$.   In particular, $\mu$ is a divisor of $|\SL_2(\ZZ/N\ZZ)/\{\pm I\}|$ and thus $\mu\leq N^3/2$.  Take any nonnegative real numbers $b$.  From our bound on $\mu$ and $N\geq 3$, we have $\mu+b\leq (1/2+b/27)N^3$.  Part (\ref{L:trivial bounds i}) is obtained by taking specific values of $b$.   

We now bound $m$.  We have $m\leq g+1$ since $\sum_{c\in \calC_G-\Sigma} w_c \geq 1$.  For every congruence subgroup $\Gamma\subseteq \SL_2(\ZZ)$ of level at most $5$, $\calX_\Gamma$ has genus $0$, see~\cite{MR2016709}.   So if $N\leq 5$, then $g=0$ and hence $m=1$; the bound $m\leq \tfrac{1}{24} N^3$ is immediate since $N\geq 3$.  We can now assume that $N\geq 6$.   By \cite[Proposition~1.40]{MR1291394}, we have $g+1\leq \mu/12+3/2$, where we have used that $X_G$ has at least one cusp.  If $\mu$ is a proper divisor of $|\SL_2(\ZZ/N\ZZ)/\{\pm I\}|$, then $\mu\leq N^3/4$ and hence
\[
m\leq g+1 \leq \tfrac{1}{48} N^3 + \tfrac{3}{2} \leq \tfrac{1}{48} N^3 + \tfrac{3}{2\cdot 6^3} N^3 \leq \tfrac{1}{24} N^3.
\]
 Finally assume that $\mu=|\SL_2(\ZZ/N\ZZ)/\{\pm I\}|$ and hence $G\cap \SL_2(\ZZ/N\ZZ)=\SL_2(\ZZ/N\ZZ)$.  In this case, we have $g=1+\mu(N-6)/(12N)$, cf.~\cite[Equation (1.6.4)]{MR1291394}.   Therefore, 
\[
m\leq g+1 \leq 2 + \tfrac{1}{2} N^3 \tfrac{N-6}{12N} \leq \tfrac{1}{24} N^3 - \tfrac{1}{4} N^2 + 2 \leq \tfrac{1}{24} N^3. \qedhere
\]
\end{proof}

\section{Proof of Theorem~\ref{T:main}} \label{S:main proof}

Take any point $P$ in $Y_G(K)$ with $j(P)\in \OO_{K,S}$.  Let $S'$ be the set of places $v$ of $K$ for which $v$ is infinite or $|j(P)|_v>1$.  We have $j(P)\in \OO_{K,S'}$ and $|S'|\leq |S| < \mathfrak{c}_{G,K}$.   So without loss of generality, we may assume that $S=S'$, i.e., a finite place $v$ of $K$ lies in $S$ if and only if $|j(P)|_v>1$.

We have $K_G\subseteq \QQ(\zeta_N)\subseteq \CC$.  Without loss of generality, we may assume that $K_G\subseteq K \subseteq \CC$.  Let $\Kbar$ be the algebraic closure of $K$ in $\CC$ and define $\Gal_K:=\Gal(\Kbar/K)$.   Define the field $L:=K(\zeta_N)$.  Let $\calC_G$ be the set of cusps in $X_G(\CC)$; we have $\calC_G\subseteq X_G(L)$ by Lemma~\ref{L:cusps of X(N)}(\ref{L:cusps of X(N) ii}).  

Take any $v\in S$; it is a place of $K$.  We choose a place of $L$ that extends $v$ which, by abuse of notation, we also denote by $v$.   This choice of place of $L$ will ultimately not matter since we are interested in $|j(P)|_v$ which does not depend on the choice of extension since $j(P)\in K$.   As in \S\ref{S:nearby}, we define a subset $\Omega_{c,v}\subseteq X_G(\Lbar_v)$ for each cusp $c\in \calC_G$.    By Lemma~\ref{L:union of neighborhoods}, there is a cusp $c_v\in \calC_G$ for which $P$ lies in $\Omega_{c_v,v}$.  Intuitively, $c_v$ is a cusp that is ``nearby'' $P$ in $X_G(\Lbar_v)$.

Let $\Sigma$ be the minimal subset of $\calC_G$ that contains $\{c_v: v\in S\}$ and is stable under the $\Gal_K$-action.   The action of $\Gal_K$ on $\Sigma$ clearly has at most $|S|$ orbits.   We have $|S|<\mathfrak{c}_{G,K}$ by  assumption and hence $\Sigma$ is a proper subset of $\calC_G$.  The set $\Sigma$ is nonempty since $S$ is nonempty.

Let $m$ be the smallest positive integer for which $m\sum_{c\in \calC_G-\Sigma} w_c > g$, where $g$ is the genus of $X_G$.   Let $\varphi \in K_G(X_G)\subseteq K(X_G)$ be a nonconstant function satisfying all the properties of Theorem~\ref{T:nouveau varphi} with respect to the set $\Sigma$ and the field $L$. 
  
Let $\Sigma_1,\ldots, \Sigma_h$ be the $\Gal_K$-orbits of $\Sigma$.  For each $1\leq i \leq h$, the values $w_c$ agree as we vary over all $c\in \Sigma_i$; we denote this common integer by $w_i$.\\

We now fix an integer $1\leq i \leq h$.   Let $S_i$ be the set of places $v\in S$ such that $c_v\in \Sigma_i$ and such that $|j(P)|_v>3500$ if $v$ is infinite.     Our goal is to find an upper bound for $\sum_{v \in S_i} d_v \log |j(P)|_v$, where the integers $d_v$ are defined in \S\ref{SS:notation}.  Once this is done we will combine these bounds for all $i$ to obtain an upper bound for $h(j(P))$.
 
The function $\varphi$ is regular at all $c\in \Sigma_i$ by Theorem~\ref{T:nouveau varphi}(\ref{T:nouveau varphi C}). Define the function
\[
g_i:=\prod_{c\in \Sigma_i}(\varphi-\varphi(c)).
\]
We now give some basic properties of $g_i$.

\begin{lemma} \label{L:properties of gi}
\begin{romanenum}
\item \label{L:properties of gi i}
The rational function $g_i$ lies in $K(X_G)$ and any pole of $g_i$ is a cusp of $X_G$ that does not lie in the set $\Sigma$.   
\item \label{L:properties of gi ii}
We have $g_i(c)=0$ for all $c\in \Sigma_i$.
\item \label{L:properties of gi iii}
We have $\varphi(P)\in \OO_{K,S}$ and $g_i(P)\in \OO_{K,S}$.
\end{romanenum}
\end{lemma}
\begin{proof}
For any $\sigma\in \Gal_K$ and $c\in \Sigma_i$, we have $\sigma(\varphi(c))=\varphi(\sigma(c))$ since $\varphi \in K(X_G)$.   Since $\Sigma_i$ is stable under the $\Gal_K$-action, the polynomial $Q_i(x):=\prod_{c\in \Sigma_i}(x-\varphi(c))$ lies in $K[x]$.   Moreover, $Q_i(x)\in \OO_K[x]$ since $\varphi(c)$ is algebraic for all $c\in \Sigma$ by Theorem~\ref{T:nouveau varphi}(\ref{T:nouveau varphi C}).    Therefore, $g_i=Q_i(\varphi)$ is an element of $K(X_G)$.   Since $g_i=Q_i(\varphi)$, any pole of $g_i$ will also be a pole of $\varphi$.  Part (\ref{L:properties of gi i}) thus follows from Theorem~\ref{T:nouveau varphi}(\ref{T:nouveau varphi B}) and (\ref{T:nouveau varphi C}).  Part (\ref{L:properties of gi ii}) is immediate from the definition of $g_i$.

We have $g_i=Q_i(\varphi)$ and hence $g_i(P)=Q_i(\varphi(P))$; the functions are regular at $P$ since $P$ is not a cusp.  Since $Q_i(x)$ is a polynomial in $\OO_K[x]$, to prove (\ref{L:properties of gi iii}) it suffices to show that $\varphi(P)$ lies in $\OO_{K,S}$.  Since $\varphi$ and $P$ are defined over $K$, we have $\varphi(P)\in K$.  By property (\ref{T:nouveau varphi A}) of Theorem~\ref{T:nouveau varphi}, $\varphi$ is the root of a monic polynomial with coefficients in $\ZZ[j]$.  Therefore, $\varphi(P)$ is the root of a monic polynomial with coefficients in $\ZZ[j(P)]\subseteq \OO_{K,S}$.   We thus have $\varphi(P)\in \OO_{K,S}$ since $\OO_{K,S}$ is integrally closed.
\end{proof}

Define the numbers 
\[
\beta:=2 (2^3 4.5^{36m}  N^{108m+15} m^{72m+1})^{m\mu+1}   4.5^{12m} N^{36m+4}
\] 
and $C:=96.6 \cdot 0.1^{24m} N^{90m+4}$.  We now bound the absolute value of $g_i(P)$ at infinite places.  

\begin{lemma} \label{L:bound on gi infinite}
For any infinite place $v$ of $K$, we have 
\[
|g_i(P)|_v \leq \begin{cases}
(\beta C)^{|\Sigma_i|} \cdot |j(P)|_v^{-1/w_i} & \text{ if $v\in S_i$,}\\
(\beta C)^{|\Sigma_i|} & \text{ otherwise}.
\end{cases}
\]
\end{lemma}
\begin{proof}
Recall we have chosen a place of $L$ extending $v$ that we also denoted by $v$.  We have also chosen a cusp $c_v\in \Sigma$ such that $P \in \Omega_{c_v,v}$.  Since $P$ is not a cusp,  Theorem~\ref{T:nouveau varphi}(\ref{T:nouveau varphi D}) implies that $|\varphi(P)-\varphi(c_v)|_v \leq  \beta C/2$.  Take any cusp $c\in \Sigma$.  By Theorem~\ref{T:nouveau varphi}(\ref{T:nouveau varphi C}), we have $|\varphi(c)|_v \leq \beta m^{24m}$.   By using the bound on $m$ from Lemma~\ref{L:trivial bounds}, one can show that $|\varphi(c)|_v \leq  \beta C/4$.

Take any $c\in \Sigma_i$.   We have $|\varphi(P)-\varphi(c)|_v \leq  |\varphi(P)-\varphi(c_v)|_v + |\varphi(c_v)|_v+|\varphi(c)|_v$ and hence $|\varphi(P)-\varphi(c)|_v \leq \beta C/2+\beta C/4+\beta C/4=\beta C$.  Therefore,
\[
|g_i(P)|_v =\prod_{c\in \Sigma_i}|\varphi(P)-\varphi(c)|_v \leq (\beta C)^{|\Sigma_i|}.
\]
Finally suppose that $v\in S_i$, i.e., $c_v\in \Sigma_i$ and $|j(P)|_v>3500$.  Note that $w_{c_v}=w_i$.  We get the other inequality of the lemma in the same manner except also using the bound $|\varphi(P)-\varphi(c_v)|_v \leq   |j(P)|_v^{-1/w_{i}} \cdot \beta C$ from Theorem~\ref{T:nouveau varphi}(\ref{T:nouveau varphi D}).
\end{proof}

We now bound the absolute value of $g_i(P)$ at finite places.

\begin{lemma} \label{L:bound on gi finite}
For any finite place $v$ of $K$, we have 
\[
|g_i(P)|_v \leq  
\begin{cases}
|j(P)|_v^{-1/w_i} & \text{ if $v\in S_i$}\\
1 & \text{ otherwise}.
\end{cases}
\]
\end{lemma}
\begin{proof}
Take any finite place $v$ of $K$.  First suppose that $v\notin S$.  We have $g_i(P)\in \OO_{K,S}$ by Lemma~\ref{L:properties of gi}(\ref{L:properties of gi iii}) and hence $|g_i(P)|_v\leq 1$.  

We can now assume that $v\in S$ and hence $|j(P)|_v>1$.    Recall we have chosen a place of $L$ extending $v$ that we also denoted by $v$.  We have also chosen a cusp $c_v\in \Sigma$ such that $P \in \Omega_{c_v,v}$.  Since $P$ is not a cusp, Theorem~\ref{T:nouveau varphi}(\ref{T:nouveau varphi D}) implies that $|\varphi(P)-\varphi(c_v)|_v \leq |j(P)|_v^{-1/w_{c_v}}$.   In particular, $|\varphi(P)-\varphi(c_v)|_v \leq 1$.

Take any cusp $c\in \Sigma$.  Since $\varphi(c)$ and $\varphi(c_v)$ are integral by Theorem~\ref{T:nouveau varphi}(\ref{T:nouveau varphi C}), we find that
\[
|\varphi(P)-\varphi(c)|_v \leq \max\{ |\varphi(P)-\varphi(c_v)|_v, |\varphi(c_v)-\varphi(c)|_v\} \leq 1.
\]
Therefore, $|g_i(P)|=\prod_{c\in \Sigma_i}|\varphi(P)-\varphi(c)|_v \leq 1$.  Now assume that $v\in S_i$ and hence $c_v\in \Sigma_i$.  Using $w_{c_v}=w_i$, we have $|g_i(P)|=\prod_{c\in \Sigma_i}|\varphi(P)-\varphi(c)|_v \leq |\varphi(P)-\varphi(c_v)|_v \leq |j(P)|_v^{-1/w_i}$.
\end{proof}

Define $C':=22.16  N^{144m+7} 0.024^{24m}$.

\begin{lemma}  \label{L:partial height Si}
We have
\[
[K:\QQ]^{-1} \sum_{v\in S_i} d_v \log |j(P)|_v \leq w_i |\Sigma_i| \log(\beta C').
\]
\end{lemma}
\begin{proof}
We first assume that $g_i(P)\in K$ is nonzero.  We have $\prod_{v\in M_K} |g_i(P)|_v^{d_v}=1$ by the product formula.  Using the upper bounds on $|g_i(P)|_v$ from Lemmas~\ref{L:bound on gi infinite} and \ref{L:bound on gi finite}, we have
\[
1\leq \prod_{v\in M_{K,\infty}}(\beta C)^{d_v|\Sigma_i|}\cdot \prod_{v\in S_i} |j(P)|_v^{-d_v/w_i}.
\]
Taking logarithms and using that $\sum_{v\in M_{K,\infty}} d_v=[K:\QQ]$ gives 
\begin{align*}
[K:\QQ]^{-1}\sum_{v\in S_i} d_v \log |j(P)|_v &\leq w_i |\Sigma_i| \log(\beta C).
\end{align*}
To prove the lemma in this case, it thus suffices to show that $C\leq C'$.  We have
\[
C'/C=\tfrac{22.16}{96.6} N^{54m+3} (\tfrac{0.024}{0.1})^{24m} \geq \tfrac{22.16\cdot 3^3}{96.6}  (3^{54}\cdot \tfrac{0.024^{24}}{0.1^{24}})^m > 6 \cdot 10^{10m} >1,
\]
where we have used that $N\geq 3$.\\

We may now assume that $g_i(P)=0$.  We have $\varphi(P)=\varphi(c)$ for some $c\in \Sigma_i$.   Since $\varphi(c)=\varphi(P)$ is in $K$, $\varphi$ is defined over $K$ and $\Sigma_i$ has a transitive $\Gal_K$-action, we find that $\varphi(P)=\varphi(c')$ for all $c'\in \Sigma_i$.  With our fields $K_G\subseteq K \subseteq L$ and $\Sigma':=\Sigma_i$, let $\xi$ be a nonzero element of $\OO_K$ as in Theorem~\ref{T:nouveau varphi}(\ref{T:nouveau varphi E}). 

Take any place $v\in S_i$.  We have $P\in \Omega_{c_v,v}-\{c_v\}$ with $c_v\in \Sigma_i$.  Moreover, $|j(P)|_v>3500$ if $v$ is infinite.  By Theorem~\ref{T:nouveau varphi}(\ref{T:nouveau varphi E}), we have $|\xi|_v\leq |j(P)|_v^{-w_i}$ if $v$ is finite and  $|\xi|_v\leq |j(P)|_v^{-w_i} \cdot (\beta C')^{w_i|\Sigma_i|}$ if $v$ is infinite.    

For any infinite place $v\notin S_i$ of $K$, we have $|\xi|_v\leq  (\beta C')^{w_i|\Sigma_i|}$ by Theorem~\ref{T:nouveau varphi}(\ref{T:nouveau varphi E}).  Since $\xi$ is integral, we have $|\xi|_v\leq 1$ for all finite places $v\notin S_i$ of $K$.  The product formula and the above inequalities give
\[
1=\prod_{v\in M_K} |\xi|_v^{d_v} \leq \prod_{v\in S_i} |j(P)|_v^{-w_i d_v} \cdot \prod_{v\in M_{K,\infty}} (\beta C')^{w_i|\Sigma_i| d_v}.
\]
By taking logarithms and using $\sum_{v\in M_{K,\infty}} d_v=[K:\QQ]$, we obtain 
\[
[K:\QQ]^{-1}\sum_{v\in S_i} d_v \log |j(P)|_v \leq |\Sigma_i| \log(\beta C') \leq w_i |\Sigma_i| \log(\beta C'). \qedhere
\]
\end{proof}

Recall that a finite place $v$ of $K$ lies in $S$ if and only if $|j(P)|_v>1$.  A place $v\in S$ lies in $S_i$ for some $1\leq i \leq h$ if $v$ is finite or $|j(P)|_v>3500$.  Therefore,
\begin{align*}
h(j(P))&=[K:\QQ]^{-1} {\sum}_{v\in M_K} d_v \log \max\{1,|j(P)|_v\} \\
&\leq [K:\QQ]^{-1} \sum_{i=1}^h \sum_{v\in S_i} d_v \log |j(P)|_v + [K:\QQ]^{-1} \sum_{v\in M_{K,\infty}} d_v \log 3500.
\end{align*}
By Lemma~\ref{L:partial height Si} and $\sum_{v\in M_{K,\infty}} d_v =[K:\QQ]$, we obtain
\begin{align*}
h(j(P)) \leq \sum_{i=1}^h w_i |\Sigma_i|  \log(\beta C') + \log 3500.
\end{align*}
Observe that $\sum_{i=1}^h w_i |\Sigma_i| \leq \sum_{c\in \Sigma} w_c \leq \mu$, where the last inequality follows from (\ref{E:sum of widths}).  Therefore,
\begin{align}\label{E:better j bound}
h(j(P))&\leq \mu \log (\beta C')   + \log 3500.
\end{align}
This is our explicit upper bound for $h(j(P))$; note that the numbers $\beta$ and $C'$ are both expressed solely in terms of $\mu$, $m$ and $N$.   

We finish by making some estimates to produce worse, though more aesthetically pleasing, bounds for $h(j(P))$.  

\begin{lemma} \label{L:mu bound for j}
We have $h(j(P))\leq 4 (\mu+4)^4 \log N$.
\end{lemma}
\begin{proof}
Using $m\leq \tfrac{1}{24} N^3$ from Lemma~\ref{L:trivial bounds}, we find that 
\begin{align*}
\beta&\leq 2(2^3 4.5^{36m} (\tfrac{1}{24})^{72m+1})^{m\mu+1} 4.5^{12m} \cdot N^{(324m + 18)(m\mu+1)+(36m+4)}\\
&= 2(\tfrac{1}{3})^{m\mu+1} (\tfrac{4.5}{24^2})^{36m(m\mu+1)}  4.5^{12m}  N^{(324m + 18)(m\mu+1)+(36m+4)}
\end{align*}
and hence 
\[
\beta C' \leq  44.32(\tfrac{1}{3})^{m\mu+1} (\tfrac{4.5}{24^2})^{36m(m\mu+1)} (4.5\cdot 0.024^2)^{12m} \cdot  N^d,
\]
where $d:=(324m + 18)(m\mu+1)+(36m+4) +(144m+7)$.  We have $\mu\geq 2$  (since otherwise $X_G$ has only one cusp and hence $1\leq |S|<\mathfrak{c}_{G,K} =1$).  Using that $m\geq 1$ and $m\mu+1 \geq 3$, we deduce that $\beta C' \leq 44.32(\tfrac{1}{3})^{3} (\tfrac{4.5}{24^2})^{108} (4.5\cdot 0.024^2)^{12} N^d$.  Taking logarithms, we find that 
$\log(\beta C')\leq d\log N -594.98$.  Since $\mu \geq 2$, we have
\[
\mu\log(\beta C')+ \log 3500 \leq \mu d\log N -2\cdot 594.98 +\log 3500 \leq \mu d \log N -1181.
\]
In particular, $\mu \log(\beta C') + \log 3500 \leq \mu d \log N$ and hence $h(j(P))\leq \mu d \log N$ by (\ref{E:better j bound}). 

By \cite[Proposition~1.40]{MR1291394}, we have $g+1\leq \mu/12+1$, where we have used that $X_G$ has at least two cusps due to the Runge condition on $X_G$.  Using $m\leq g+1\leq \mu/12+1$, we obtain an upper bound for $\mu d$ that is a polynomial in $\mu$.  In particular, $\mu d \leq f(\mu)$, where $f(x)=(9x^4 + 222x^3 + 1536x^2 + 2132x)/4$.   One can readily check that $f(x) \leq 4 (x+4)^4$ holds for all $x\geq 0$.  Therefore, $\mu d\leq 4 (\mu+4)^4$ and hence $h(j(P))\leq \mu d \log N \leq 4 (\mu+4)^4 \log N$.
\end{proof}

By Lemma~\ref{L:trivial bounds}, we have $\mu\leq N^3/2$ and hence $h(j(P))\leq 4(N^3/2+4)^4 \log N$ by Lemma~\ref{L:mu bound for j}. Using that $N\geq 3$, one can check that $4(N^3/2+4)^4 \leq N^{12}$ and hence $h(j(P))\leq N^{12} \log N$.

\section{Background: modular forms} \label{S:modular forms}

In this section, we recall some facts about modular forms.  For the basics on modular forms see \cite{MR1291394}.   

\subsection{Notation} \label{SS:modular form setup}

Recall that the group $\SL_2(\ZZ)$ acts by linear fractional transformations on the complex upper half-plane $\calH$ and the extended upper half-plane $\calH^*=\calH\cup \QQ \cup \{\infty\}$.       Consider an integer $k\geq 0$.  For a meromorphic function $f$ on $\calH$ and a matrix $\gamma=\left(\begin{smallmatrix}a & b \\c & d\end{smallmatrix}\right)\in \SL_2(\ZZ)$, define the meromorphic function $f|_k \gamma$ on $\calH$ by  $(f|_k \gamma)(\tau):= (c\tau+d)^{-k} f(\gamma \tau)$; we call this the \defi{slash operator} of weight $k$.

For a congruence subgroup $\Gamma$ of $\SL_2(\ZZ)$, we denote by $M_k(\Gamma)$ the set of \defi{modular forms} of weight $k$ on $\Gamma$; it is a finite dimensional complex vector space. Recall that each $f\in M_k(\Gamma)$ is a holomorphic function of the upper half-plane $\calH$ that satisfies $f|_k \gamma=f$ for all $\gamma\in \Gamma$ along with the familiar growth condition at the cusps.  For each modular form $f\in M_k(\Gamma)$, we have 
\[
f(\tau) = \sum_{n=0}^\infty a_n(f)\, q_w^{n}
\]
for unique $a_n(f)\in \CC$, where $w$ is the width of the cusp $\infty$ of $\Gamma$ and $q_w:=e^{2\pi i \tau/w}$.  We call this power series in $q_w$, the \defi{$q$-expansion} of $f$ (at the cusp $\infty$).   Note that $f$ is uniquely determined by its $q$-expansion and hence we can identify $f$ with its $q$-expansion.   For a subring $R$ of $\CC$, we denote by $M_k(\Gamma,R)$ the $R$-submodule of $M_k(\Gamma)$ consisting of modular forms whose $q$-expansion has coefficients in $R$.  

\subsection{Actions on $M_k(\Gamma(N))$} \label{SS:actions on Mk}

Fix positive integers $N$ and  $k$.  Since $\Gamma(N)$ is normal in $\SL_2(\ZZ)$, the slash operator of weight $k$ defines a right action of $\SL_2(\ZZ)$ on $M_k(\Gamma(N))$.   Take any modular form $f=\sum_{n=0}^\infty a_n(f) q_N^n$ in $M_k(\Gamma(N))$.  For every field automorphism $\sigma$ of $\CC$, there is a unique modular form $\sigma(f) \in  M_k(\Gamma(N))$ whose $q$-expansion is $\sum_{n=0}^\infty \sigma(a_n(f))\, q_N^{n}$.   This defines an action of $\Aut(\CC)$ on $M_k(\Gamma)$.   

The following lemma says that the above actions of $\SL_2(\ZZ)$ and $\Aut(\CC)$ induce a right action of $\GL_2(\ZZ/N\ZZ)$ on $M_{k}(\Gamma(N))$.  For each $d\in (\ZZ/N\ZZ)^\times$, let $\sigma_d$ be any automorphism of $\CC$ for which $\sigma_d(\zeta_N)=\zeta_N^d$.

\begin{lemma} \label{L:star action}
There is a unique right action $*$ of $\GL_2(\ZZ/N\ZZ)$ on $M_k(\Gamma(N),\QQ(\zeta_N))$ such that the following hold:
\begin{itemize}
\item
if $A \in \SL_2(\ZZ/N\ZZ)$, then  $f*A = f|_k \gamma$, where $\gamma$ is any matrix in $\SL_2(\ZZ)$ that is congruent to $A$ modulo $N$,
\item
if $A=\left(\begin{smallmatrix}1 & 0 \\0 & d\end{smallmatrix}\right)$, then $f*A=\sigma_d(f)$.
\end{itemize}
\end{lemma}
\begin{proof}
See \cite[\S3]{BN2019}.
\end{proof}

\subsection{The spaces $M_{k,G}$} \label{SS:MkG}

Fix a positive integer $N$ and let $G$ be a subgroup of $\GL_2(\ZZ/N\ZZ)$.   For each integer $k\geq 0$, define
\[
M_{k,G}:=M_k(\Gamma(N),\QQ(\zeta_N))^G,
\]
i.e., the subgroup fixed by the $G$-action $*$ from Lemma~\ref{L:star action}.   Note that $M_{k,G}$ is a vector space over $K_G=\QQ(\zeta_N)^{\det G}$.     The following lemma explains how we will use modular forms to produce regular functions on $X_G$.
 
 \begin{lemma} \label{L:quotient of modular forms}
Fix an integer $k \geq 0$ and take any modular forms $f$ and $g$ in $M_{k,G}$ with $g\neq 0$.  Then $f/g$ is an element of $K_G(X_G)=\calF_N^G$.
 \end{lemma}
 \begin{proof}
For any $\gamma\in \Gamma(N)$, we have $(f/g)(\gamma \tau) = f(\gamma \tau)/g(\gamma\tau) = (f|_k\gamma)(\tau)/(g|_k \gamma)(\tau)=f(\tau)/g(\tau)$.  Therefore, $f/g$ is a modular function of level $N$.   We have $f/g \in \calF_N$ since the $q$-expansions of $f$ and $g$ both have coefficients in $\QQ(\zeta_N)$.

For any $A\in G$, we have $(f/g)*A = (f*A)/(g*A)=f/g$, where the first $*$-action is the one from Lemma~\ref{L:Shimura basic}.   Therefore, $f/g$ is an element of $\calF_N^G=K_G(X_G)$.
\end{proof}

Take any cusp $c \in \calC_G$ of $X_G$.  Choose an $A\in \SL_2(\ZZ/N\ZZ)$ for which $A\cdot c_\infty=c$, where $c_\infty$ is the cusp at infinity.  Set $w=w_c$.  For any $f \in M_{k,G}$, the $q$-expansion of $f*A$ lies in $\QQ(\zeta_N)[\![q_{w}]\!]$.    When $f$ is nonzero, we define $\nu_c(f)$ to be the minimal integer $n$ for which the coefficient of $q_w^n$ in the $q$-expansion of $f*A$ is nonzero.   When $f$ is zero, we define $\nu_c(f)=+\infty$.  Note that $\nu_c(f)$ does not depend on the choice of $A$.   Now consider any $f,g\in M_{k,G}$ with $g$ nonzero.   We have $f/g \in K_G(X_G) \subseteq \QQ(\zeta_N)(X_G)$ by Lemma~\ref{L:quotient of modular forms}.  Moreover,
\begin{align} \label{E:ordc}
\ord_c(f/g)=\nu_c(f)-\nu_c(g).
\end{align}
   
 \subsection{Eisenstein series of weight 1} \label{SS:weight 1}

See \cite[\S3]{MR2104361} for the basics on Eisenstein series.  For further information, we refer to \S\S2--3 of \cite{BN2019} where all the results below are summarized and referenced (except for the explicit constant $c_0$ in Lemma~\ref{L:Eisenstein expansion}, see \cite[Lemma~3.1]{MR3705252} instead).   We will restrict our attention to weight $1$ modular forms; our functions $E_\alpha$ are denoted $E_\alpha^{(1)}$ in \cite{BN2019}.

Fix a positive integer $N$.  Consider any pair $\alpha\in (\ZZ/N\ZZ)^2$ and choose $a,b\in \ZZ$ with $\alpha\equiv (a,b) \pmod{N}$.  With $\tau\in \calH$, consider the series
\begin{align} \label{E:Ealpha def}
E_\alpha(\tau,s) = \frac{1}{-2\pi i} 
\sum_{\substack{\omega\in \ZZ + \ZZ\tau\\\omega \neq - (a\tau+b)/N}}  \Big(\frac{a\tau+b}{N} + \omega\Big)^{-1}
\cdot \Big| \frac{a\tau+b}{N} + \omega\Big|^{-2s}.
\end{align}
The series (\ref{E:Ealpha def}) converges when the real part of $s\in \CC$ is large enough.   Hecke proved that $E_\alpha(\tau,s)$ can be analytically continued to all $s\in \CC$.    Using this analytic continuation, we define the \defi{Eisenstein series} 
\[
E_{\alpha}(\tau) :=E_\alpha(\tau,0).
\]
For $\gamma\in \SL_2(\ZZ)$, we have $E_{\alpha} |_1 \gamma = E_{\alpha \gamma}$, where $\alpha \gamma \in (\ZZ/N\ZZ)^2$ denotes matrix multiplication.  In particular, $E_\alpha$ is fixed by $\Gamma(N)$ under the slash operator of weight $1$.

\begin{lemma} \label{L:Eisenstein expansion}
Take any $a,b\in \ZZ$ and let $\alpha \in (\ZZ/N\ZZ)^2$ be the image of $(a,b)$ modulo $N$.
The function $E_\alpha$ is a modular form of weight $1$ on $\Gamma(N)$ with $q$-expansion
\[
c_0 + \sum_{\substack{m,n \geq 1\\ m\equiv a \bmod{N}}} \zeta_N^{bn} q_N^{mn} -  \sum_{\substack{m,n \geq 1\\ m\equiv -a \bmod{N}}}  \zeta_N^{-bn} q_N^{mn}
\]
and $c_0 \in \QQ(\zeta_N)$.  Moreover,
\[
c_0= \begin{cases}
     0 & \text{if $a\equiv b \equiv 0\pmod{N}$}, \\
     \frac{1}{2} \, \frac{1+\zeta_N^b}{1-\zeta_N^b} & \text{if $a\equiv 0 \pmod{N}$ and $b\not\equiv 0 \pmod{N}$},\\
     \frac{1}{2} -  \frac{a_0}{N} & \text{if $a\not\equiv 0 \pmod{N}$,}     
\end{cases}
\]
where $0\leq a_0<N$ is the integer congruent to $a$ modulo $N$.
\end{lemma}

The right action $*$ of $\GL_2(\ZZ/N\ZZ)$ on $M_1(\Gamma(N),\QQ(\zeta_N))$, described in \S\ref{SS:actions on Mk}, acts on the Eisenstein series $E_\alpha$ in a pleasant manner.

\begin{lemma} \label{L:essential EaA}
We have $E_\alpha * A = E_{\alpha A}$ for all $\alpha\in (\ZZ/N\ZZ)^2$ and $A\in \GL_2(\ZZ/N\ZZ)$.
\end{lemma}

\subsection{Expressing modular forms in terms of Eisenstein series} \label{SS:KM}

Using the Eisenstein series of weight $1$ from \S\ref{SS:weight 1}, we can generate higher weight modular forms.  

\begin{thm}[Khuri-Makdisi] \label{T:Eisenstein span}
Suppose $N>2$.   The $\CC$-subalgebra of $\bigoplus_{k\geq 0} M_k(\Gamma(N))$ generated by the Eisenstein series $E_\alpha$ with $\alpha\in (\ZZ/N\ZZ)^2$  contains all modular forms of weight $k$ on $\Gamma(N)$ for all $k\geq 2$. 
\end{thm}
\begin{proof}
 This particular formulation of results of Khuri-Makdisi \cite{MR2904927} is Theorem~3.1 of \cite{BN2019}.
\end{proof}

The following is a direct consequence of the above theorem.  It describes an explicit finite set of generators for $M_{k,G}$ as a vector space over $\QQ$. 

\begin{lemma} \label{L:KM reinterpreted}
Fix integers $N>2$ and $k\geq 2$.  Let $G$ be a subgroup of $\GL_2(\ZZ/N\ZZ)$.   Then $M_{k,G}$, as a $\QQ$-vector space, is spanned by the set of modular forms of the form 
\begin{align} \label{E:is a trace}
\sum_{A\in G} \zeta_N^{j \det A}\, E_{\alpha_1 A} \cdots E_{\alpha_k A}
\end{align}
with $\alpha_i \in (\ZZ/N\ZZ)^2-\{0\}$ and $0\leq j < [\QQ(\zeta_N):\QQ]$.
\end{lemma}
\begin{proof}
Let $S$ be the set of modular forms of the form $\zeta_N^j E_{\alpha_1} \cdots E_{\alpha_k}$ with pairs $\alpha_1,\ldots, \alpha_k \in (\ZZ/N\ZZ)^2-\{(0,0)\}$ and an integer $0\leq j <\phi(N):=[\QQ(\zeta_N):\QQ]$.  Since $E_{(0,0)}=0$, Theorem~\ref{T:Eisenstein span} implies that $S$ spans the complex vector space $M_k(\Gamma(N))$.  We have $S\subseteq M_k(\Gamma(N),\QQ(\zeta_N))$ by Lemma~\ref{L:Eisenstein expansion}.   Since $k\geq 2$ and $N>2$, the natural map $M_k(\Gamma(N),\QQ(\zeta_N))\otimes_{\QQ(\zeta_N)} \CC \to M_k(\Gamma(N))$ is an isomorphism of complex vector spaces, cf.~\cite[\S1.7]{MR0447119}.     Since $S\subseteq M_k(\Gamma(N),\QQ(\zeta_N))$ spans $M_k(\Gamma(N))$, we deduce that $S$ also spans the $\QQ(\zeta_N)$-vector space $M_k(\Gamma(N),\QQ(\zeta_N))$.  Since $1,\zeta_N, \ldots, \zeta_N^{\phi(N)-1}$ is a basis of $\QQ(\zeta_N)$ as a vector space over $\QQ$, we find $S$ further spans $M_k(\Gamma(N),\QQ(\zeta_N))$ as a vector space over $\QQ$.

Define the $\QQ$-linear map $T\colon M_k(\Gamma(N),\QQ(\zeta_N)) \to M_{k,G}$, $f\mapsto \sum_{A\in G} f*A$.  The map $T$ is surjective since we have $T(f)=|G|\,f$ for all $f\in M_{k,G}$.   Therefore, $M_{k,G}$ as a $\QQ$-vector space is spanned by the set of $T(f)$ with $f\in S$.   It remains to compute $T(f)$ for $f\in S$.  Take any $f=\zeta_N^j E_{\alpha_1} \cdots E_{\alpha_k}\in S$. We have 
\begin{align*}
T(f)= \sum_{A\in G} (\zeta_N^j E_{\alpha_1} \cdots E_{\alpha_k})*A =\sum_{A\in G} \zeta_N^{j\det A} (E_{\alpha_1}*A) \cdots (E_{\alpha_k}*A).
\end{align*}
Finally note that $T(f)$ equals (\ref{E:is a trace}) by Lemma~\ref{L:essential EaA}.
\end{proof}

\begin{remark}
In \cite{OpenImage}, under the extra assumption $\det(G)=(\ZZ/N\ZZ)^\times$, a version of Lemma~\ref{L:KM reinterpreted} is required for an algorithm which finds an explicit basis of $M_{k,G}$ for any given even integer $k\geq 2$.   Using a suitable weight $k\in \{2,4,6\}$, this basis can then used to compute an explicit model of the curve $X_G$ over $K_G=\QQ$.
\end{remark}

\section{A basis of modular forms with relatively small coefficients} \label{S:basis}

Fix an integer $N>2$ and let $G$ be a subgroup of $\GL_2(\ZZ/N\ZZ)$ that satisfies $-I \in G$.  In \S\ref{SS:MkG}, we defined a finite dimensional $K_G$-vector space $M_{k,G}$ of modular forms of level $N$ and weight $k$.   Since $-I \in G$, we find that $M_{k,G}=0$ for $k$ odd.  

In this section, we prove the following result which shows that there is basis of $M_{k,G}$, viewed as a vector over $\QQ$, consisting of modular forms whose $q$-expansion at each cusp has integral and relatively small coefficients.

\begin{thm} \label{T:small basis}
Fix an even integer $k\geq 2$.   Then there is a basis $f_1,\ldots, f_d$ of the $\QQ$-vector space $M_{k,G}$ such that for every $1\leq i \leq d$ and every $A\in \GL_2(\ZZ/N\ZZ)$, we have
\[
f_i * A = \sum_{n=0}^\infty a_n q_N^n,
\]
where each coefficient $a_n$ lies in $\ZZ[\zeta_N]$ and satisfies 
\[
|a_n|_v \leq 2|G| 4.5^k N^{k} \max\{N^k,n^{2k}\}
\]
for all infinite places $v$ of $\QQ(\zeta_N)$.
\end{thm}

\subsection{Bounding coefficients}

Fix an even integer $k\geq 2$.  Lemma~\ref{L:KM reinterpreted} shows that there is a basis of $M_{k,G}$ obtained from Eisenstein series of weight $1$.   We will want to bound the size of the Fourier coefficients that arise in such a basis.    We first prove a simple lemma that will be needed for these bounds.  For each positive integer $n$, let $d(n)$ be the number of positive divisors of $n$.

\begin{lemma} \label{L:Sjn}
Take any positive integers $j$ and $n$, and define the sum
\[
S_{j,n}:=\sum_{a_1+\cdots+a_j=n} \, \prod_{i=1}^j d(a_i)
\]
where $a_1,\ldots, a_j$ vary over all positive integers.  Then $S_{j,n}  \leq 2n^{j-1/2}(\log n +1)^{j-1}$.
\end{lemma}
\begin{proof}
We proceed by induction on $j\geq 1$.  We have $d(n)\leq 2n^{1/2}$ since for every positive divisor $e$ of $n$ at least one of $e$ or $n/e$ is bounded above $n^{1/2}$.  The case $j=1$ is now clear since $S_{1,n}=d(n)$.

Now consider any $j\geq 1$ for which the lemma holds.  We have $S_{j+1,n} = \sum_{a=1}^n d(a) S_{j,n-a}$.   Thus by our inductive hypothesis, we have $S_{j+1,n} \leq 2 n^{j-1/2}(\log n +1)^{j-1} \sum_{a=1}^n d(a)$.  So to prove the lemma, it suffices to show that $\sum_{a=1}^n d(a)\leq n (\log n + 1)$.

We have inequalities $\sum_{a=1}^n d(a) = \#\{(b,c)\in \NN: bc \leq n\} \leq \sum_{b=1}^{n} \lfloor n/b \rfloor \leq n \sum_{b=1}^n 1/b$. Therefore, $\sum_{a=1}^n d(a) \leq n(1+ \sum_{b=2}^n 1/b)\leq n(1+\int_1^n 1/x \, dx)=n(\log n +1)$ as desired.
\end{proof}

\begin{lemma} \label{L:original Eisenstein bounds}
Fix positive integers $k$ and $N$.  For any $\alpha_1,\ldots, \alpha_k \in  (\ZZ/N\ZZ)^2$, the $q$-expansion of $E_{\alpha_1}\cdots E_{\alpha_k}$ is of form $\sum_{n=0}^\infty a_n q_N^n$ with $a_n\in \QQ(\zeta_N)$ that satisfy $|a_0| \leq (N/4)^k$ and 
\[
|a_n| \leq 2n^{-1/2} (N/4+2n(\log n+1))^k
\]
for all $n\geq 1$.
\end{lemma}
\begin{proof}
First take any $\alpha \in (\ZZ/N\ZZ)^2$ and fix $a,b\in \ZZ$ so that $\alpha\equiv (a,b) \pmod{N}$ with $0\leq a < N$ and $|b|\leq N/2$. Let $\sum_{n=0}^\infty c_n(\alpha) q_N^n$ be the $q$-expansion of $E_\alpha$.  We have $c_n(\alpha)\in \QQ(\zeta_N)$ by Lemma~\ref{L:Eisenstein expansion}.  Moreover, the explicit $q$-expansion given in Lemma~\ref{L:Eisenstein expansion} implies that $|c_n(\alpha)| \leq 2 d(n)$ for all $n\geq 1$.  

We claim that $|c_0(\alpha)|\leq N/4$.   From Lemma~\ref{L:Eisenstein expansion}, we have $|c_0(\alpha)| \leq 1/2$ if $a\not\equiv 0 \pmod{N}$ or $b\equiv 0 \pmod{N}$.  So to prove the claim, we may assume from Lemma~\ref{L:Eisenstein expansion} that $b\not \equiv 0 \pmod{N}$ and $c_0(\alpha)=1/2\cdot (1+\zeta_N^b)/(1-\zeta_N^b)$.   We have $|c_0(\alpha)|^2 \leq |1-\zeta_N^b|^{-2}=(2-2\cos(\theta))^{-1}$, where $\theta:=2\pi b/N \in [-\pi,\pi]$.   Since $2-2\cos(x) \geq 4/\pi^2 \cdot x^2$ for all real $x\in [-\pi,\pi]$ and $\theta \neq 0$, we deduce that $|c_0(\alpha)|^2 \leq \pi^2/4 \cdot \theta^{-2}=N^2/(16b^2) \leq N^2/16$.  Therefore, $|c_0(\alpha)| \leq N/4$ as claimed.

Multiplying the $q$-expansions of $E_{\alpha_1}, \ldots, E_{\alpha_k}$ together, we have
\begin{align} \label{E:Eisenstein an explicit}
a_n = \sum_{n_1+\cdots+n_k = n,\, n_i\geq 0} \,\prod_{i=1}^k c_{n_i}(\alpha_i).
\end{align}
For $n=0$, we have $|a_0| = \prod_{i=1}^k |c_0(\alpha_i)| \leq (N/4)^k$ which proves the lemma in this case.

Now take any integer $n\geq 1$.  For each subset $I \subseteq \{1,\ldots, k\}$, we can consider those terms in the sum (\ref{E:Eisenstein an explicit}) for which we have $n_i=0$ exactly when $i\in I$. Using our bounds for the coefficients $c_n(\alpha)$, we deduce that
\[
|a_n| \leq \sum_{j=1}^k \binom{k}{j} \cdot (N/4)^{k-j}  \sum_{a_1+\cdots + a_j=n, a_i \geq 1}\, \prod_{i=1}^j 2d(a_i).
\]
Equivalently,  $|a_n| \leq \sum_{j=1}^k \binom{k}{j} \cdot (N/4)^{k-j} \cdot 2^j S_{j,n}$ with $S_{j,n}$ as in Lemma~\ref{L:Sjn}.  By Lemma~\ref{L:Sjn}, we obtain the upper bounds
\[
|a_n| \leq 2 \sum_{j=1}^k \binom{k}{j} \cdot (N/4)^{k-j} \cdot 2^j n^{j-1/2} (\log n +1)^j \leq 2n^{-1/2} \sum_{j=1}^k \binom{k}{j} (N/4)^{k-j}(2n(\log n +1))^j.
\]
So by the binomial theorem, we have $|a_n| \leq 2n^{-1/2} (N/4+2n(\log n+1))^k$.
\end{proof}

\begin{lemma} \label{L:estimate for product of Eisenstein series at all cusps}
Fix positive integers $k$ and $N$ and take any $\alpha_1,\ldots, \alpha_k \in  (\ZZ/N\ZZ)^2$.   Define the modular form $h:=(2N)^k E_{\alpha_1}\cdots E_{\alpha_k} \in M_k(\Gamma(N),\QQ(\zeta_N))$.  For any $A \in \GL_2(\ZZ/N\ZZ)$, the $q$-expansion of $h*A$ is of form $\sum_{n=0}^\infty a_n q_N^n$ so that each coefficient $a_n$ is an element of $\ZZ[\zeta_N]$
that satisfy $|a_0| \leq N^{2k}/2^k$ and 
\[
|a_n| \leq 2n^{-1/2} N^k (N/2+4n(\log n+1))^k
\]
for all $n\geq 1$.
\end{lemma}
\begin{proof}
By Lemma~\ref{L:Eisenstein expansion}, we find that $h$, and hence also $h*A$, are modular forms in $M_k(\Gamma(N),\QQ(\zeta_N))$.   Using Lemma~\ref{L:essential EaA}, we have $h*A=(2N)^k (E_{\alpha_1}*A)\cdots (E_{\alpha_k}*A)= (2N)^k E_{\beta_1}\cdots E_{\beta_k}$, where $\beta_i:=\alpha_i A$.  Therefore, the desired bounds on the $|a_n|$ follow from those of Lemma~\ref{L:original Eisenstein bounds} after multiplying by $(2N)^k$.    

Finally, to prove that each $a_n$ lies in $\ZZ[\zeta_N]$, it suffices to show that the $q$-expansion of $2N\cdot E_{\beta_i}$ has coefficients in $\OO_{\QQ(\zeta_N)}=\ZZ[\zeta_N]$ for all $1\leq i \leq k$. From the explicit $q$-expansions given in Lemma~\ref{L:Eisenstein expansion}, it suffices to show that $N/(1-\zeta_N^b)$ lies in $\ZZ[\zeta_N]$ for any integer $b\not\equiv 0 \pmod{N}$.  We have $N/(1-\zeta_N^b) \in \ZZ[\zeta_N]$ since $\zeta_N^{b}-1$ is a root of $((x+1)^N-1)/x=x^{N-1}+\cdots + N \in \ZZ[x]$.  
\end{proof}

\subsection{Proof of Theorem~\ref{T:small basis}}

By Lemma~\ref{L:KM reinterpreted}, there is a basis $f_1,\ldots, f_d$ of the $\QQ$-vector space $M_{k,G}$ such every $f_i$ has a $q$-expansion of the form
\begin{align} \label{E:trace final}
(2N)^k \sum_{g\in G} \zeta_N^{j \det g} E_{\alpha_1g} \cdots E_{\alpha_k g}
\end{align}
for some $\alpha_1,\ldots, \alpha_k \in (\ZZ/N\ZZ)^2$ and integer $j$.   

Take any $1\leq i \leq d$ and $A\in \GL_2(\ZZ/N\ZZ)$.   Let $\sum_{n=0}^\infty a_n q_N^n$ be the $q$-expansion of $f_i*A$.    Take any infinite place $v$ of $\QQ(\zeta_N)$.   Since $\QQ(\zeta_N)\subseteq \CC$, there is a $\sigma\in \Gal(\QQ(\zeta_N)/\QQ)$ such that $|a|_v=|\sigma(a)|$ for all $a\in \QQ(\zeta_N)$.  There is a unique $d'\in (\ZZ/N\ZZ)^\times$ for which $\sigma(\zeta_N)=\zeta_N^{d'}$.  Define $B:=A\left(\begin{smallmatrix}1 & 0 \\0 & d'\end{smallmatrix}\right) \in \GL_2(\ZZ/N\ZZ)$.  By Lemma~\ref{L:star action}, we have $f_i*B = \sum_{n=0}^\infty \sigma(a_n) q_N^n$.   Since the $q$-expansion of $f_i$ is of the form (\ref{E:trace final}), Lemma~\ref{L:estimate for product of Eisenstein series at all cusps} implies that all of the $\sigma(a_n)$ are in $\ZZ[\zeta_N]$ and 
\[
|a_n|_v=|\sigma(a_n)|\leq 
\begin{cases}
|G| \cdot N^{2k}/2^k & \text{if $n=0$,}\\
 |G| \cdot 2n^{-1/2} N^k (N/2+4n(\log n+1))^k & \text{if $n\geq 1$.}
 \end{cases}
 \]
It remains to prove that $ |a_n|_v \leq 2|G| 4.5^k N^{k} \max\{N^k,n^{2k}\}$.  This is immediate for $n=0$ from the above bound, so assume that $n\geq 1$.  We have $\log n + 1 \leq n$, so $|a_n|_v \leq 2 |G|N^k (N/2+4n^2)^k$.  When $N \leq n^2$, we have $|a_n|_v \leq 2 |G| N^k (4.5 n^2)^k$.   When $N\geq n^2$, we have $|a_n|_v \leq 2|G| N^k (4.5N)^k$.  The theorem is now immediate.

\section{Riemann--Roch} \label{S:RR}
 
 Fix an integer $N>2$ and let $G$ be a subgroup of $\GL_2(\ZZ/N\ZZ)$ with $-I\in G$.  Let $g$ be the genus of the curve $X_G$ and let $\mu$ be the degree of the morphism $j\colon X_G \to \PP^1_{K_G}$.   The group $\Gal_{K_G}$ acts on the set of cusps $\calC_G$ of $X_G$ by Lemma~\ref{L:cusps of X(N)}(\ref{L:cusps of X(N) ii}).
 
Consider a proper subset of $\Sigma\subseteq \calC_G$ that is stable under the $\Gal_{K_G}$-action.  For our application, we will require a nonconstant function $\varphi\in K_G(X_G)$ whose poles are all at cusps and has no pole at any cusp $c\in \Sigma$.  In this section, we use the Riemann--Roch theorem to find spaces of modular forms from which we can construct such $\varphi$.

  Consider a divisor $D=\sum_{c\in \calC_G} n_c \cdot c$ of $X_G$ that is defined over $K_G$, i.e., $n_c=n_{\sigma(c)}$ for all $c\in \calC_G$ and $\sigma\in \Gal_{K_G}$.  Define the $K_G$-vector space $\scrL(D):=\{\varphi \in K_G(X_G) : \operatorname{div}(\varphi)+D\geq 0\}$; equivalently, $\scrL(D)$ consists of those functions $\varphi\in K_G(X_G)$ whose poles all occur at cusps and $\ord_c(\varphi)+n_c \geq 0$ for all $c\in \calC_G$.  By the Riemann--Roch theorem, we have $\dim_{K_G} \scrL(D) \geq \deg(D)-g+1$ with equality holding when $\deg(D)> 2g-2$.\\
  
 For a fixed positive integer $m$, define the divisors $D_0:=\sum_{c\in \calC_G} mw_c\cdot c$ and $D_1:=\sum_{c\in \calC_G-\Sigma} mw_c\cdot c$.  The divisors $D_0$ and $D_1$ are defined over $K_G$ since $\calC_G$ and $\Sigma$ are both stable under the $\Gal_{K_G}$-action (also $w_{\sigma(c)}=w_c$ for all $c\in \calC_G$ and $\sigma \in \Gal_{K_G}$). The functions in $\scrL(D_0)$ have no poles away from the cusps.  The functions in the subspace $\scrL(D_1)\subseteq \scrL(D_0)$ are regular at all $c\in \Sigma$.
 
By Lemma~\ref{L:quotient of modular forms}, we have an injective $K_G$-linear map
\[
T\colon M_{12m,G} \hookrightarrow K_G(X_G),\quad f\mapsto f/\Delta^m.
\]
Let $W_m$ be the $K_G$-subspace of $M_{12m,G}$ consisting of those modular forms $f$ for which $\nu_{c}(f) \geq mw_c$ for all $c\in \Sigma$, where $\nu_{c}$ is defined in \S\ref{SS:MkG}.

\begin{lemma} \label{L:T image}
\begin{romanenum}
\item
We have $T(M_{12m,G})=\scrL(D_0)$ and $T(W_m)=\scrL(D_1)$.  
\item
We have $\dim_{K_G} M_{12m,G}=m\mu-g+1$ and 
\[
\dim_{K_G} W_m \geq m\sum_{c\in \calC_G-\Sigma} w_c -g+1.
\]
\end{romanenum}
\end{lemma}
\begin{proof}
We first prove that $T(M_{12m,G})=\scrL(D_0)$.   Take any $f\in M_{12m,G}$ and define $\varphi:=T(f)=f/\Delta^m\in K_G(X_G)$.   Every pole of $\varphi$ is a cusp since $\Delta$, when viewed as a function of the upper half-plane, is holomorphic and everywhere nonzero.   Take any cusp $c\in \calC_G$.  By (\ref{E:ordc}), we have $\ord_c \varphi = \nu_{c}(f)-\nu_{c}(\Delta^m)=\nu_{c}(f)-m w_c \geq -mw_c$.  Therefore, $T(f)=\varphi \in \scrL(D_0)$.  We have $T(M_{12m,G})\subseteq\scrL(D_0)$ since $f$ was an arbitrary element of $M_{12m,G}$.

Now take any $\varphi\in \scrL(D_0)$.  Define $f:=\varphi \Delta^m$; it is a weakly modular form of weight $12m$.  For each $c\in \calC_G$, we have $\ord_c(\varphi) + \nu_{c}(\Delta^m)=\ord_c(\varphi) + mw_c \geq 0$, where the inequality use our choice of $\varphi$.  Therefore, $f$ is a modular form of weight $12m$ on $\Gamma(N)$.   We have $f \in M_{12m}(\Gamma(N),\QQ(\zeta_N))$ since the $q$-expansions of $\varphi$ and $\Delta$ have coefficients in $\QQ(\zeta_N)$.    For any $A\in G$, we have $f*A=(\varphi \Delta^m)*A=(\varphi*A) (\Delta^m *A)=\varphi \Delta^m$.  Therefore, $f\in M_{12m,G}$ and hence $\varphi=T(f)$ lies in $T(M_{12m,G})$.   We have $T(M_{12m,G})\supseteq\scrL(D_0)$ since $\varphi$ was an arbitrary element of $\scrL(D_0)$.   This completes the proof that $T(M_{12m,G})=\scrL(D_0)$.

For any $f\in M_{12m,G}$, we have $\ord_c(T(f))= \nu_{c}(f)-mw_c$ for each $c\in \calC_G$.  So $T(f)\in \scrL(D_1)$ if and only if $\nu_{c}(f)\geq mw_c$ for each $c\in \Sigma$.  Since $T(M_{12m,G})=\scrL(D_0)$, we deduce that $T(W_m)=\scrL(D_1)$.   

By the Riemann--Roch theorem, we have \[ \dim_{K_G} W_m = \dim_{K_G} \scrL(D_1) \geq \deg(D_1)-g+1= m\sum_{c\in \calC-\Sigma}w_c -g +1.\]  We have $\deg(D_0)=\sum_{c\in \calC_G} mw_c = m\mu$ by (\ref{E:sum of widths}) and $2g-2<\mu/6$ by \cite[Proposition~1.40]{MR1291394}.  Therefore, $\deg(D_0)>2g-2$.  By the Riemann--Roch theorem, we have $\dim_{K_G} M_{12m,G} = \dim_{K_G} \scrL(D_0) = \deg(D_1)-g+1= m\mu-g+1$.
\end{proof}
 
Take any $f\in W_m$ and define $\varphi:=T(f)=f/\Delta^m$.  By Lemma~\ref{L:T image}, $\varphi \in K_G(X_G)$ is a function that is regular at all $c\in \Sigma$ and  all of its poles are at cusps.   In order for $\varphi$ to be nonconstant, we need $f$ not to lie in the subspace $K_G\Delta^m\subseteq W_m$.     Thus to construct a nonconstant $\varphi$ in this manner, we need $\dim_{K_G} W_m \geq 2$; by Lemma~\ref{L:T image} this holds for all sufficiently large $m$.

\section{Existence of a certain modular form} \label{S:existence certain}

Fix an integer $N>2$ and let $G$ be a subgroup of $\GL_2(\ZZ/N\ZZ)$ with $-I\in G$.   Let $g$ be the genus of the curve $X_G$ and let $\mu$ be the degree of the morphism $j\colon X_G \to \PP^1_{K_G}$.   

Let $\Sigma$ be a proper subset of $\calC_G$ that is stable under the $\Gal_{K_G}$-action.  Let $m$ be the smallest positive integer for which $m\sum_{c\in \calC_G-\Sigma} w_c > g$.     Define the real number 
\[
\beta:=2 (2^3 4.5^{36m}  N^{108m+15} m^{72m+1})^{m\mu+1}   4.5^{12m} N^{36m+4}.
\]

\begin{thm} \label{T:modular form of weight k}
 There is a nonzero modular form $f \in M_{12m,G}$ that satisfies the following properties:
\begin{alphenum}
\item  \label{T:modular form of weight k a}
We have $\nu_{c}(f) \geq mw_c$ for all $c\in \Sigma$. 
\item  \label{T:modular form of weight k b}
The modular form $f$ does not lie in $K_G \Delta^m$.
\item  \label{T:modular form of weight k c}
Take any $A\in \SL_2(\ZZ/N\ZZ)$.  Define the cusp $c:=A \cdot c_\infty$ of $X_G$, where $c_\infty$ is the cusp at infinity, and set $w=w_c$.  Then we have $f*A=\sum_{n=0}^\infty b_n q_w^n$,
where each coefficient $b_n$ lies in $\ZZ[\zeta_N]$ and satisfies 
\[
|b_n|_v \leq  \beta \max\{1,(n/w)^{24m}\} 
\]
for each infinite place $v$ of $\QQ(\zeta_N)$.
\end{alphenum}
\end{thm}

\subsection{Proof of Theorem~\ref{T:modular form of weight k}}
   
Set $k=12m$ and define $d:=\dim_\QQ M_{k,G}$.  Fix a basis $f_1,\ldots, f_d$ of the $\QQ$-vector space $M_{k,G}$ satisfying the conclusion of Theorem~\ref{T:small basis}.   Using Lemma~\ref{L:T image}, we find that
\[
d=[K_G:\QQ] \dim_{K_G} M_{k,G} = [K_G:\QQ](m\mu-g+1).
\]

Let $W$ be the $K_G$-subspace of $M_{k,G}$ consisting of those modular forms $f$ for which $\nu_{c}(f)\geq mw_c$ for all $c\in \Sigma$.  By Lemma~\ref{L:T image} and our choice of $m$, we have $\dim_{K_G} W \geq 2$.  We will now describe $W$ as the kernel of a linear map on $M_{k,G}$.

Define $L:=\QQ(\zeta_N)$.  For each cusp $c\in \Sigma$, we choose a matrix $A\in \SL_2(\ZZ/N\ZZ)$ for  which $A\cdot c_\infty=c$.  For any $f\in M_{k,G}$, we have  $f*A=\sum_{n=0}^\infty a_{c,n}(f)\, q_{w_c}^n$ with $a_{c,n}(f)\in L$.  Define the $\QQ$-linear map 
\[
\psi_c\colon M_{k,G}\to L^{mw_c},\quad f\mapsto (a_{c,0}(f),\ldots, a_{c,mw_c-1}(f)).
\]  
Combining the maps $\psi_c$ with $c\in \Sigma$, we obtain a $\QQ$-linear map
\[
\psi\colon M_{k,G}\to \prod_{c\in \Sigma} L^{mw_c}.
\]
Note that the kernel of $\psi$ is $W$.  Let
\[
\alpha_0\colon \QQ^d \to  \prod_{c\in \Sigma} L^{mw_c} 
\]
be the $\QQ$-linear map obtained by composing the isomorphism $\QQ^d\xrightarrow{\sim} M_{k,G}$ coming from the basis $f_1,\ldots,f_d$ with $\psi$.   We have $\dim_\QQ \ker (\alpha_0) = \dim_\QQ W$.

We have an isomorphism $L\otimes_\QQ \RR \xrightarrow{\sim} \prod_{v\in M_{L,\infty}} L_v$ of real vector spaces induced by the inclusions $L\subseteq L_v$.  Tensoring $\alpha_0$ with $\RR$ gives a $\RR$-linear map 
\[
\alpha\colon V_1\to V_2,
\] 
where $V_1=\RR^d$ and $V_2=\prod_{c\in \Sigma} \prod_{v\in M_{L,\infty}} L_v^{mw_c}$.   We have $\dim_\RR \ker (\alpha) = \dim_\QQ W$.

We define norms on $V_1$ and $V_2$ by $\norm{b}_1:=\sum_{i=1}^d |b_i|$ and $\norm{(b_{c,v})}_2:=\max_{c\in \Sigma, v\in M_{L,\infty}} |b_{c,v}|_v$, respectively.     The group $M_1:=\ZZ^d$ is a lattice in $V_1$; it is generated by elements of norm $1$.     Let $M_2$ be the lattice in $V_2$ corresponding to the subgroup $\prod_{c\in \Sigma} \OO_L^{mw_c}$ of $\prod_{c\in \Sigma} L^{mw_c}$.  For each $i\in \{1,2\}$ and nonzero $b\in M_i$, we have $\norm{b}_i\geq 1$.  For each $1\leq i \leq d$, we have $\psi_c(f_i) \in \OO_L^{mw_c}$ for all $c\in \Sigma$ by our choice of basis $f_1,\ldots, f_d$.  Therefore, $\alpha(M_1) \subseteq M_2$.    

The \emph{norm} $\norm{\alpha}$ of $\alpha$ is the minimal real number for which $\norm{\alpha(v)}_2 \leq \norm{\alpha} \norm{v}_1$ holds for all $v\in V_1$.  We now find an upper bound for $\norm{\alpha}$.

\begin{lemma} \label{L:bound on norm of alpha}
We have $\norm{\alpha} \leq 2|G| 4.5^k  N^{3k} m^{2k}$.
\end{lemma}
\begin{proof}
Take any $b\in \RR^d$.   We have $\alpha(b)=\sum_{i=1}^d b_i \alpha(e_i)$, where $e_1,\ldots, e_d$ is the standard basis of $\RR^d$.   Taking the norm gives $\norm{\alpha(b)}_2 \leq \sum_{i=1}^d |b_i| \norm{\alpha(e_i)}_2 \leq \max_{1\leq i \leq d} \norm{\alpha(e_i)}_2 \cdot \norm{b}_1$.  Therefore, $\norm{\alpha}\leq \max_{1\leq i \leq d} \norm{\alpha(e_i)}_2$ and hence
\[
\norm{\alpha} \leq \max \{ |a_{c,n}(f_i)|_v : i \in \{1,\ldots, d\}, n \in \{0,\ldots, mw_c-1\}, c\in \Sigma, v\in M_{L,\infty}\}.  
\]

Take any $1\leq i \leq d$, $c\in \Sigma$ and $v\in M_{L,\infty}$.  Recall that $f*A=\sum_{n=0}^\infty a_{c,n}(f) q_{w_c}^n$ for some $A\in \SL_2(\ZZ/N\ZZ)$.   In particular, $f*A=\sum_{n=0}^\infty a_{c,n}(f) q_{N}^{nN/w_c}$.  By Theorem~\ref{T:small basis}, we have $|a_{c,0}(f_i)|_v \leq |G| N^{2k}/2^{k}$.   Now take any $1\leq n <mw_c$ and define the integer $n_0:=nN/w_c$.  We have $1\leq n_0 < mN$.  By Theorem~\ref{T:small basis}, we have
\begin{align*}
|a_{c,n}(f_i)|_v &\leq 2|G| 4.5^k  N^{k} \max\{N^k,n_0^{2k}\}\leq 2|G| 4.5^k  N^{3k} m^{2k}.
\end{align*}
Combining everything together, we have now shown that $\norm{\alpha} \leq 2|G| 4.5^k  N^{3k} m^{2k}$.
\end{proof}

Before proceeding, we recall the following version of Siegel's lemma due to Faltings.

\begin{lemma} \label{L:Siegel-Faltings}
Let $V_1$ and $V_2$ be finite dimensional real vector spaces with norms $|\cdot|_1$ and $|\cdot|_2$, respectively.  
Let $M_1$ and $M_2$ be $\ZZ$-lattices of $V_1$ and $V_2$, respectively.  Let $\alpha\colon V_1\to V_2$ be a linear map that satisfies $\alpha(M_1)\subseteq M_2$.   Let $C\geq 2$ be a real number such that $\alpha$ has norm at most $C$ (i.e., $|\alpha(v)|_2\leq C\cdot |v|_1$ for all $v\in V_1$), $M_1$ is generated by elements of norm at most $C$, and every nonzero element of $M_1$ and $M_2$ has norm at least $C^{-1}$.   Define $a=\dim \ker(\alpha)$ and $b=\dim V_1$.  Then for each $0\leq i \leq a-1$, $\ker(\alpha)$ contains linearly independent elements $m_1,\ldots, m_{i+1}$ of $M_1$ satisfying 
\[
\max_{1\leq j \leq i+1} |m_j|_1 \leq (C^{3}\cdot b)^{b/(a-i)}.
\]
\end{lemma}
\begin{proof}
See \cite[Proposition~2.18]{MR1109353} or \cite[Lemma 4]{MR1289008}; they give the upper bound $(C^{3b}\cdot b!)^{1/(a-i)}$ from which our follows by using that $b!\leq b^b$.
\end{proof}

We now apply Siegel's lemma in our setting.   Define $\calB:=(2^3 4.5^{3k}  N^{9k+15} m^{6k+1})^{m\mu+1}$.

\begin{lemma} \label{L:good u}
There is a $u\in \ZZ^d$ with $\norm{u}_1 \leq \calB$ such that $\sum_{i=1}^d u_i f_i$ lies in $W$ but does not lie in the subspace $K_G\Delta^m$.
\end{lemma}
\begin{proof}
Define $a:=\dim_{\RR} \ker(\alpha)=\dim_\QQ W = [K_G:\QQ]\dim_{K_G} W \geq 2[K_G:\QQ]$ and $b:=\dim V_1 = d = [K_G:\QQ](m\mu-g+1)$.   Using Lemma~\ref{L:bound on norm of alpha}, we find that the conditions of Lemma~\ref{L:Siegel-Faltings} hold in our setting with $C:=2|G| 4.5^k  N^{3k} m^{2k}$.

Now take $i:=[K_G:\QQ]$; we have $i<a$.   By Lemma~\ref{L:Siegel-Faltings}, there are linearly independent vectors $m_1,\ldots, m_{i+1}\in \ZZ^d$ in $\ker (\alpha)$ satisfying 
$\norm{m_j}_1 \leq (C^3 b)^{b/(a-i)}$ for all $1\leq j \leq i+1$.  Since $i+1>[K_G:\QQ]=\dim_{\QQ}(K_G\Delta^m)$, there is a vector $u\in \{m_1,\ldots, m_{i+1}\} \subseteq \ZZ^d$ so that 
\[
f:=\sum_{i=1}^d u_i f_i
\]
is a modular form in $M_{k,G}$ for which $f$ does not lie in $K_G\Delta^m$.   Since $u$ is in the kernel of $\alpha$, and hence also the kernel of $\alpha_0$, we find that $f$ lies in $W$.   We have $\norm{u}_1\leq (C^3b)^{b/(a-i)}$ since $u=m_j$ for some $j$.  

It remains to bound $(C^3b)^{b/(a-i)}$ from above.  We have $a-i \geq 2[K_G:\QQ]-[K_G:\QQ]=[K_G:\QQ]$ and $b=[K_G:\QQ] (m\mu-g+1)$, so $b/(a-i) \leq m\mu-g+1 \leq m\mu+1$.   We have
\[
C^3b = 2^3 4.5^{3k}  N^{9k} m^{6k}  \cdot  |G|^3 [K_G:\QQ] \cdot (m\mu-g+1).
\]
We have $[K_G:\QQ]\leq [\GL_2(\ZZ/N\ZZ):G]$, so $|G|^3 [K_G:\QQ]\leq |\GL_2(\ZZ/N\ZZ)|^3 \leq N^{12}$.  We have $m\mu-g+1 \leq m\mu+1 \leq 2m \mu  \leq m N^3$ by Lemma~\ref{L:trivial bounds}.  Therefore, $C^3b \leq 2^3 4.5^{3k}  N^{9k+15} m^{6k+1}$ and hence $(C^3b)^{b/(a-i)}\leq \calB$. 
\end{proof}

Fix a $u\in \ZZ^d$ as in Lemma~\ref{L:good u}.   Define $f:=\sum_{i=1}^d u_i f_i$; it is an element of $W$ that does not lie in $K_G \Delta^m$.

Take any matrix $A\in \SL_2(\ZZ/N\ZZ)$ and take any infinite place $v$ of $\QQ(\zeta_N)$   Define the cusp $c:=A\cdot c_\infty$ and set $w=w_c$.   We now consider the $q$-expansions $f*A=\sum_{n=0}^\infty a_n q_N^n$ and $f_i*A=\sum_{n=0}^\infty a_{i,n} q_N^n$ with $1\leq i \leq d$.  Since $f=\sum_{i=1}^d u_i f_i$, we have $a_n=\sum_{i=1}^d u_i a_{i,n}$ for all $n\geq 0$.   So for $n\geq 0$, we have $|a_n|_v \leq \norm{u}_1 \max_{1\leq i \leq d} |a_{i,n}|_v \leq \calB \max_{1\leq i \leq d} |a_{i,n}|_v$. 

Since $f_1,\ldots, f_d$ is a basis of $M_{k,G}$ satisfying the conclusion of Theorem~\ref{T:small basis}, we find that
\[
|a_n|_v \leq  \calB \cdot 2\cdot 4.5^k N^{k+4}  \max\{N^k,n^{2k}\}
\]
where we have used the bound $|G|\leq N^4$.  However, recall that we are interested in the coefficients of the $q$-expansion $\sum_{n=0}^\infty b_n q_w^n$ of $f*A$.   We must have $b_n=a_{nN/w}$ for all $n\geq 0$.  So for $n\geq 0$, we have
\begin{align*}
|b_n|_v =|a_{nN/w}|_v
&\leq \calB \cdot  2\cdot 4.5^k N^{k+4}  \max\{N^k,(nN/w)^{2k}\} \\
&\leq \calB \cdot  2\cdot 4.5^k N^{2k+4}  \max\{1,N^k(n/w)^{2k}\}\\
&\leq \calB \cdot  2\cdot 4.5^k N^{3k+4}  \max\{1,(n/w)^{2k}\}.
\end{align*}
The theorem follows since $k=12m$ and hence $\beta=\calB \cdot  2\cdot 4.5^k N^{3k+4}$.

\subsection{Some bounds}

We now prove some technical lemmas that we will later use to bound functions near cusps.   Note that the bound on $c_n$ in the following lemma is, up to a constant factor, that of $|b_n|_v$ from Theorem~\ref{T:modular form of weight k}.

\begin{lemma} \label{L:sum of cn}
Take any real number $0\leq u \leq e^{-\pi\sqrt{3}}$ and positive integer $m$.  Fix a positive divisor $w$ of $N$ and let $\{c_n\}_{n\geq 1}$ be a sequence of nonnegative real numbers that satisfy $c_n \leq \max\{1,(n/w)^{24m}\}$.
\begin{romanenum}
\item \label{L:sum of cn i}
For any integer $B\geq 5mw$, we have $\sum_{n=B}^\infty c_n u^{n/w} \leq 230.8 w u^{B/w} (B/w)^{24m+1}$.
\item \label{L:sum of cn ii}
For any integer $mw\leq B\leq 5mw$, we have $\sum_{n=B}^\infty c_n u^{n/w}\leq 231.6 w u^{B/w} (5m)^{24m+1}$.
\end{romanenum}
\end{lemma}
\begin{proof}
Set $u_0:=e^{-\pi\sqrt{3}}$ and $a:=-24/\log(u_0)\cdot m=m\cdot 4.410\ldots$.  Define the function $g(x)=x^{24m} u_0^x$.  One can check that $g(x)$ is increasing for $0\leq x \leq a$ and decreasing for $x\geq a$.   

We first assume that $B\geq 5mw$.   We have
\begin{align*}
\sum_{n=B}^\infty c_n u^{n/w} = u^{B/w} \sum_{n=0}^\infty c_{n+B} u^{n/w}
= u^{B/w} \sum_{e=0}^\infty \Big(\sum_{r=0}^{w-1} c_{ew+r+B} \,u^{r/w}\Big) u^e.
\end{align*}
For all integers $e\geq 0$ and $0\leq r < w$, we have
\[
c_{ew+r+B} \leq ((ew+r+B)/w)^{24m} \leq (e+1+B/w)^{24m},
\]
where the first bound uses that $(ew+r+B)/w\geq B/w \geq 1$.  Using that $u^{r/w}\leq 1$ for all $0\leq r < w$ and that $u\leq u_0$, we obtain 
\begin{align} \label{E:early cn sum}
\sum_{n=B}^\infty c_n u^{n/w} \leq  u^{B/w} w \sum_{e=0}^\infty (e+1+B/w)^{24m} u_0^e = u^{B/w} w u_0^{-1-B/w} C,
\end{align}
where $C:=\sum_{e=0}^\infty g(e+1+B/w)$.  Since $g(x)$ is decreasing for $x\geq 5m$ and $B/w \geq 5m$, we have
\[
C \leq \int_{0}^\infty g(x+B/w) \, dx = \int_{B/w}^\infty g(x)\, dx = \sum_{i=0}^\infty s_i,
\]  
where $s_i:= \int_{2^iB/w}^{2^{i+1}B/w} g(x)\, dx$.  For $x\geq 5m$, we have $g(2x)/g(x)=2^{24m} u_0^x \leq (2^{24} u_0^5)^m\leq 0.000026$.
For each $i\geq 1$, we have
\[
s_i = 2\int_{2^{i-1}B/w}^{2^{i}B/w} g(2x)\, dx \leq 0.000026 \int_{2^{i-1}B/w}^{2^{i}B/w} g(x)\, dx = 0.000026 s_{i-1}.
\]
Therefore, $s_i \leq (0.000026)^i s_0$ for all $i\geq 0$ and hence $C \leq \sum_{i=0}^\infty s_i \leq (1-0.000026)^{-1} s_0 \leq 1.00003s_0$.   Since $g(x)$ is decreasing for $x\geq 5m$ and $B/w\geq 5m$, we have $s_0 \leq g(B/w) B/w$ and hence $C \leq 1.00003 (B/w)^{24m+1} u_0^{B/w}$.  By (\ref{E:early cn sum}), we deduce that 
\[
\sum_{n=B}^\infty c_n u^{n/w}
\leq 1.00003u_0^{-1} w u^{B/w} (B/w)^{24m+1}\leq 230.8 w u^{B/w} (B/w)^{24m+1}
\]
which proves (\ref{L:sum of cn i}).

We now assume that $mw\leq B < 5mw$.    From part (\ref{L:sum of cn i}), in the special case $B=5mw$, we have
\begin{align} \label{E:case B eq 5mw}
\sum_{n=5mw}^\infty c_n u^{n/w} 
\leq 230.8 w u^{5m} (5m)^{24m+1} \leq 230.8 w u^{B/w} (5m)^{24m+1}.
\end{align}
We have 
\begin{align*}
\sum_{n=B}^{5mw-1} c_n u^{n/w} &\leq u^{B/w} \sum_{n=B}^{5mw-1} (n/w)^{24m} u^{(n-B)/w} \leq u^{B/w} (5m)^{24m} (5mw -mw).
\end{align*}
So $\sum_{n=B}^{5mw} c_n u^{n/w} \leq 0.8 w u^{B/w} (5m)^{24m+1}$ and by combining with (\ref{E:case B eq 5mw}) we deduce (\ref{L:sum of cn ii}).
\end{proof}

\begin{lemma} \label{L:Deltam sum bounds}
Take any positive integer $m$ and let $\sum_{n=m}^\infty a_n q^n$ be the $q$-expansion of $\Delta^m$.  Take any real number $0\leq u \leq e^{-\pi\sqrt{3}}$.
\begin{romanenum}
\item \label{L:Deltam sum bounds i}
We have $|a_n|\leq 2 n^{6m}$ for all $n\geq 2$.
\item \label{L:Deltam sum bounds ii}
For any integer $B> 2m$, we have $\sum_{n=B}^\infty |a_n| u^n\leq 463  u^{B} (B-1)^{6m+1}$.
\item \label{L:Deltam sum bounds iii}
For any integer $m<B\leq 2m$, we have $\sum_{n=B}^\infty |a_n| u^n\leq 465  u^{B} (2m)^{6m+1}$.
\end{romanenum}
\end{lemma}
\begin{proof}
In \cite{MR2471957}, Rouse describes a constant $C_m$ for which $|a_n| \leq C_m d(n) n^{(12m-1)/2}$ holds for all $n\geq m$; this follows from expressing $\Delta^m$ as a linear combination of Hecke eigenforms and using bounds of Deligne.   When $m>1$, an explicit upper bound for $C_m$ is given in the proof of Theorem~1 in \cite{MR2471957} from which we find that $C_m\leq 1$.  When $m=1$, we have $C_m=1$ as noted in the appendix of  \cite{MR2471957}.  Therefore, $|a_n| \leq  d(n) n^{(12m-1)/2}$.    We obtain $|a_n| \leq 2 n^{6m}$ by using the easy bound $d(n) \leq 2\sqrt{n}$.  This proves (\ref{L:Deltam sum bounds i}).

Define $u_0:=e^{-\pi\sqrt{3}}$.   Define the function $g(x)=x^{6m} u_0^x$; it is decreasing for all $x\geq a:=-6/\log(u_0) \cdot m =m\cdot 1.102\ldots$.

We first assume that $B\geq 2m$.   Using (\ref{L:Deltam sum bounds i}) and $u\leq u_0$, we have
\begin{align}\label{E:early Delta}
\sum_{n=B+1}^\infty |a_n| u^n \leq 2 u^{B+1}\sum_{n=B+1}^\infty n^{6m} u_0^{n-B-1} = 2 u^{B+1} u_0^{-B-1} C,
\end{align}
where $C:=\sum_{n=B+1}^\infty g(n)$.  Since $g(x)$ is decreasing for $x\geq B\geq 2m$, we have $C \leq \int_B^\infty g(x) \, dx = \sum_{i= 0}^\infty s_i$, where $s_i:=\int_{2^i B}^{2^{i+1}B} g(x) \, dx$.   For $x\geq 2m$, we have $g(2x)/g(x)=2^{6m} u_0^x \leq (2^{6} u_0^2)^m\leq 0.0013$.
For each $i\geq 1$, we have
\[
s_i = 2\int_{2^{i-1}B}^{2^{i}B} g(2x)\, dx \leq 0.0013 \int_{2^{i-1}B}^{2^{i}B} g(x)\, dx = 0.0013 s_{i-1}.
\]
Therefore, $s_i \leq (0.0013)^i s_0$ for all $i\geq 0$ and hence $C \leq \sum_{i=0}^\infty s_i \leq (1-0.0013)^{-1} s_0 \leq 1.0014s_0$.  Since $g(x)$ is decreasing for $x\geq B\geq  2m$, we have $s_0 \leq g(B) B$ and hence $C \leq 1.0014 B^{6m+1} u_0^{B}$.  By (\ref{E:early Delta}), we deduce that 
\[
\sum_{n=B+1}^\infty |a_n| u^n
\leq 2\cdot 1.0014 u_0^{-1}  u^{B+1} B^{6m+1} \leq 463  u^{B+1} B^{6m+1}.
\]
This implies (\ref{L:Deltam sum bounds ii}) after replacing $B$ by $B-1$.   

Now take any $m< B\leq  2m$.     We have
\begin{align*}
\sum_{n=B}^{2m} |a_n| u^n \leq u^{B} \sum_{n=B}^{2m} |a_n| \leq  u^{B} \cdot 2(2m)^{6m} \cdot (2m-B+1)  \leq   u^{B} (2m)^{6m+1}.
\end{align*}
From (\ref{L:Deltam sum bounds ii}), we have $\sum_{n=2m+1}^\infty |a_n| u^n \leq 463  u^{2m+1} (2m)^{6m+1}$.
By adding these sums, we obtain $\sum_{n=B}^\infty |a_n| u^n \leq 464  u^{B} (2m)^{6m+1}$.
\end{proof}

\section{Proof of Theorem~\ref{T:nouveau varphi}} \label{S:nouveau proof}

With $k:=12m$, fix a modular form $f\in M_{k,G}=M_{12m,G}$ as in Theorem~\ref{T:modular form of weight k}.  Define $\varphi:=f/\Delta^m$.  We have $\varphi\in K_G(X_G)$ by Lemma~\ref{L:quotient of modular forms}.  We now state some basic properties of $\varphi$.

\begin{lemma} \label{L:varphi basic}
\begin{romanenum}
\item \label{L:varphi basic i}
The function $\varphi\in K_G(X_G)$ is nonconstant and has no poles away from the cusps of $X_G$.  
\item \label{L:varphi basic ii}
The $q$-expansion of $\varphi*A$ has coefficients in $\ZZ[\zeta_N]$ for each $A\in \GL_2(\ZZ/N\ZZ)$.
\item \label{L:varphi basic iii}
For any cusp $c$ of $X_G$, we have $\ord_c(\varphi)\geq -mw_c$.  
\item  \label{L:varphi basic iv}
For any cusp $c\in \Sigma$, $\varphi$ is regular at $c$.  Moreover, $\varphi(c)$ lies in $\ZZ[\zeta_N]$ and satisfies $|\varphi(c)|_v  \leq \beta m^{24m}$ for any infinite places $v$ of $L$.
\item \label{L:varphi basic v}
The rational function $\varphi$ of $X_G$ has at most $m\mu$ poles counted with multiplicity.
\end{romanenum}
\end{lemma}
\begin{proof}
With notation as in \S\ref{S:RR}, we have $f\in W_m$ by our choice of $f$ and hence $\varphi \in \scrL(D_1)$ by Lemma~\ref{L:T image}, where $D_1:=\sum_{c\in \calC_G-\Sigma} mw_c\cdot c$.  Since $\varphi \in \scrL(D_1)$, the function $\varphi$ has no poles away from the cusps, $\ord_c \varphi \geq 0$ for all $c\in \Sigma$, and $\ord_c \varphi \geq -m w_c$ for all $c\in \calC_G$.  The function $\varphi$ is nonconstant since $f\notin K_G\Delta^m$ by our choice of $f$.  We have proved (\ref{L:varphi basic i}) and (\ref{L:varphi basic iii}).  By (\ref{L:varphi basic i}) and (\ref{L:varphi basic iii}), the number of poles of $\varphi$ is bounded above by $\sum_{c\in \calC_G} mw_c=m\mu$, where the equality uses (\ref{E:sum of widths}). This proves (\ref{L:varphi basic v}). 

We now prove (\ref{L:varphi basic ii}).  Since $\Delta^m$ lies in $q^{m}\cdot(1+q\ZZ[\![q]\!])$ and $\varphi*A=(f*A)/(\Delta^m*A)=(f*A)/\Delta^m$, it suffices to show that the $q$-expansion of $f*A$ has coefficients in $\ZZ[\zeta_N]$ for any $A\in \GL_2(\ZZ/N\ZZ)$.  Take any $A\in \SL_2(\ZZ/N\ZZ)$.  Define the cusp $c:=A\cdot c_\infty$ of $X_G$, where $c_\infty$ is the cusp at infinity, and set $w=w_c$.   By property (\ref{T:modular form of weight k c}) of Theorem~\ref{T:modular form of weight k}, we have $f*A=\sum_{n=0}^\infty b_n q_w^n$ with $b_n \in \ZZ[\zeta_N]$.   For any $B=\left(\begin{smallmatrix}1 & 0 \\0 & d\end{smallmatrix}\right) \in \GL_2(\ZZ/N\ZZ)$,  we have $f*(AB)=(f*A)*B=\sum_{n=0}^\infty \sigma_d(b_n) q_w^n \in \ZZ[\zeta_N][\![q_w]\!]$ by Lemma~\ref{L:star action}.  This completes the proof of (\ref{L:varphi basic ii}) since any matrix in $\GL_2(\ZZ/N\ZZ)$ is of the form $AB$ for some $A\in \SL_2(\ZZ/N\ZZ)$ and $B=\left(\begin{smallmatrix}1 & 0 \\0 & d\end{smallmatrix}\right)$.

We finally prove (\ref{L:varphi basic iv}).  Take any $c\in \Sigma$; we have already shown that $\varphi$ is regular at $c$.  Fix $A\in \SL_2(\ZZ/N\ZZ)$ for which $A\cdot c_\infty=c$ and set $w=w_c$.  We have $\nu_{c}(f)\geq mw$, so $f*A = \sum_{n=mw}^\infty b_n q_w^n$ with $b_n\in \ZZ[\zeta_N]$.   We have $\Delta^m \in q^{m}\cdot(1+q\ZZ[\![q]\!])=q_w^{mw}\cdot(1+q\ZZ[\![q]\!])$, so the constant term of the $q$-expansion of $\varphi=f*A/\Delta^m$ is $b_{mw}$.   Therefore, $\varphi(c)=b_{mw}$ and in particular $\varphi(c)\in \ZZ[\zeta_N]$.  For any infinite place $v$ of $L$, we have $|\varphi(c)|=|b_{mw}|\leq \beta m^{2k}= \beta m^{24m}$ by Theorem~\ref{T:modular form of weight k}(\ref{T:modular form of weight k c}).
\end{proof}

From the previous lemma, we have  proved (\ref{T:nouveau varphi B}) and (\ref{T:nouveau varphi C}).   The  following lemma proves (\ref{T:nouveau varphi A}).

\begin{lemma}
The function $\varphi$ is the root of a monic polynomial with coefficients in $\ZZ[j]$.
\end{lemma}
\begin{proof}
Define the polynomial $Q(x)=\prod_{A\in R}(x- \varphi*A)$, where $R$ is a set of representatives of the right $G$-cosets of $\GL_2(\ZZ/N\ZZ)$.     Since $\varphi$ is fixed by the right action of $G$ on $\calF_N$, we deduce that $Q(x)$ is independent of the choice of $R$ and its coefficients lie in $\calF_N^{\GL_2(\ZZ/N\ZZ)}=\QQ(j)$.  

By Lemma~\ref{L:varphi basic}(\ref{L:varphi basic i}), the function $\varphi$ has all of its poles at cusps.  So for each $A\in R$, $\varphi*A$ has all of its poles at cusps.  The coefficients of $Q(x)$ lie in $\QQ(j)$ and have all of their poles at cusps as well.  Therefore, the coefficients of $Q(x)$ lie in $\QQ[j]$.   The coefficients of $Q(x)$ all have $q$-expansions with integral coefficients by Lemma~\ref{L:varphi basic}(\ref{L:varphi basic ii}).  The lemma follows by noting that an element of $\QQ[j]$ lies in $\ZZ[j]$ if and only if its $q$-expansion has integral coefficients (to see this use that $j=q^{-1}+744+196884q + \cdots \in \ZZ(\!(q)\!)$).
\end{proof}

Now take any cusp $c \in \calC_G$ and set $w=w_c$.   Choose an $A\in \SL_2(\ZZ/N\ZZ)$ for which $A\cdot c_\infty=c$.    By property (\ref{T:modular form of weight k c}) of Theorem~\ref{T:modular form of weight k}, we have $f*A=\sum_{n=0}^\infty b_n q_w^n$ with $b_n \in \ZZ[\zeta_N]$.    
We also have a $q$-expansion $\Delta^m=q^m\prod_{i=1}^\infty (1-q^n)^{24m}=\sum_{n=m}^\infty a_n q^n$ with $a_n\in \ZZ$.  

Note that $\Delta^{-1}=q^{-1} h(q)$ with $h(x):=\prod_{n=1}^\infty (1-x^n)^{-24} \in \ZZ[\![x]\!]$.  For later, we remark that a numerical computation shows that
\begin{align} \label{E:easy h}
|h(t^N)|_v \leq  \prod_{n=1}^\infty (1-e^{-\pi\sqrt3 n})^{-24} < 1.1104
\end{align}
holds for any infinite place $v$ of $L$ and any $t\in\Lbar_v$ with $|t|_v\leq e^{-\pi\sqrt{3}/N}$.  We now show that with respect to a place $v$ of $L$, $\varphi$ can be given by an explicit $v$-analytic expression on a neighborhood of $c$.

\begin{lemma} \label{L:final t}
Take any place $v$ of $L$ and point $P\in \Omega_{c,v}-\{c\}$.  Then there is a nonzero $t\in \Lbar_v$ such that all the following hold:
\begin{itemize}
\item
We have $|t|_v<1$.  If $v$ is infinite, then $|t|_v\leq e^{-\pi \sqrt{3}/N}$. 
\item
We have 
\begin{align} \label{E:varphiA}
\varphi(P)=t^{-mN} h(t^N)^m \sum_{n=0}^\infty b_n t^{nN/w}
\end{align}
in $\Lbar_v$.
\item
If $v$ is finite, we have $|j(P)|_v=|t|_v^{-N}$.  
\item
If $v$ is infinite and $|j(P)|_v>3500$, we have $|t|_v^{N} \leq 2|j(P)|_v^{-1}$. 
\end{itemize}
\end{lemma}
\begin{proof}
Denote by $\pi\colon X(N)_{\QQ(\zeta_N)}\to (X_G)_{\QQ(\zeta_N)}$ the natural morphism corresponding to the inclusion $\QQ(\zeta_N)(X_G)\subseteq \calF_N$ of function fields.  By our definition of $\Omega_{c,v}$, there is a cusp $c'$ of $X(N)$ and a point $P'\in \Omega_{c',v}$ such that $\pi(c')=c$ and $\pi(P')=P$.  We have $P\neq c'$ since otherwise  $P=c$.  Since $\varphi \in K_G(X_G)\subseteq \QQ(\zeta_N)(X(N))$, we can view $\varphi$ as a rational function on $X(N)$ (equivalently, we denote $\varphi\circ \pi$ by $\varphi$ as well).  We have $\varphi(P)=\varphi(P')$ and $j(P)=j(P')$.   So without loss of generality, we may assume that $c$ is a cusp of $X(N)$, $P\in \Omega_{c,v}-\{c\}\subseteq X(N)(\Lbar_v)$ and we may view $\varphi$ as a function in $\QQ(\zeta_N)(X(N))=\calF_N$. 

We have 
\begin{align} \label{E:varphiA 2}
\varphi*A=\Delta^{-m}(f*A)=q^{-m}h(q)^m \sum_{n=0}^\infty b_n q_w^n= q_N^{-mN}h(q^N)^m \sum_{n=0}^\infty b_n q_N^{nN/w}
\end{align}
which when expanded is the $q$-expansion of $\varphi$.   We claim that the radius of convergence of (\ref{E:varphiA}), viewed as a power series in $\Lbar_v[\![q_N]\!]$, is at least $1$.  Since the coefficients of $h$ and all the $b_n$ are integral, the claim is immediate when $v$ is finite.   When $v$ is infinite, the claim follows from the bounds on $|b_n|_v$ given by property (\ref{T:modular form of weight k c}) of Theorem~\ref{T:modular form of weight k}.

For our fixed $A$, let $\psi_{A,v}\colon B_v\to X(N)(\Lbar_v)$ be the continuous map from Proposition~\ref{P:specializations}.  By the definition of $\Omega_{c,v}$, we have $P=\psi_{A,v}(t)$ for some nonzero $t \in B_v$ that also satisfies $|t|_v\leq e^{-\pi\sqrt{3}/N}$ when $v$ is infinite.  The $v$-analytic expression (\ref{E:varphiA}) thus holds by (\ref{E:varphiA 2}) and Proposition~\ref{P:specializations}.  By Proposition~\ref{P:specializations}, we also have $j(P)=t^{-N}+744+196884t^N+21493760t^{2N} +\cdots$ in $\Lbar_v$, where the coefficients are from the familiar $q$-expansion of $j$.   When $v$ is finite this implies that $|j(P)|_v=|t|_v^{-N}$.   When $v$ is infinite and $|j(P)|_v>3500$, we have $|t|_v^{N} \leq 2|j(P)|_v^{-1}$, cf.~\cite[Corollary~2.2]{BiluParent}.
\end{proof}

We now assume that the cusp $c=A\cdot c_\infty$ lies in $\Sigma$.   By Lemma~\ref{L:varphi basic}(\ref{L:varphi basic iv}), $\varphi$ is regular at $c$ and $\varphi(c)\in \ZZ[\zeta_N]$.  The function $\varphi-\varphi(c)$ lies in $\QQ(\zeta_N)(X_G)$ and has a zero at $c$.  

Fix an integer $1\leq r \leq \ord_c(\varphi-\varphi(c))$.   Since $A\in \SL_2(\ZZ/N\ZZ)$, we have $(\varphi-\varphi(c))*A=\varphi*A-\varphi(c)=(f*A)/\Delta^m-\varphi(c)$.  Since $r \leq \ord_c(\varphi-\varphi(c))$, the $q$-expansion of $(f*A)/\Delta^m-\varphi(c)$ lies in $q_w^{r} \cdot\QQ(\zeta_N)[\![q_w]\!]$ and hence the $q$-expansion of $f*A-\varphi(c) \Delta^m$ lies in $q_w^{mw+r} \cdot\QQ(\zeta_N)[\![q_w]\!]$. 
 Therefore,
 \begin{align} \label{E:qexpansion with cancel}
 f*A -\varphi(c) \Delta^m = \sum_{n=mw+r}^\infty b_{n}q_w^n -\varphi(c) \sum_{n\geq m+r/w}a_{n} q^{n}.
 \end{align}

Take any  place $v$ of $L$ and take any point $P\in \Omega_{c,v}-\{c\}$ that satisfies $|j(P)|_v>3500$ when $v$ is infinite.  Let $t\in \Lbar_v$ be an nonzero element satisfying the conclusions of Lemma~\ref{L:final t}.   We have
\begin{align} \label{E:varphiP raw t}
\varphi(P)-\varphi(c)&=t^{-mN} h(t^N)^m \sum_{n=0}^\infty b_n t^{nN/w} - \varphi(c)\\
\notag &= t^{-mN} h(t^N)^m\bigg(\sum_{n=mw+r}^\infty b_{n}t^{nN/w} -\varphi(c) \sum_{n\geq m+r/w}a_{n} t^{nN}\bigg),
\end{align}
where we have used $\Delta^{-m}=q^{-m}h(q)^{m}$ along with (\ref{E:qexpansion with cancel}) which ensures early terms of our sums will cancel.  Therefore,
\begin{align} \label{E:texpansion with cancel}
\varphi(P)-\varphi(c)&= t^{rN/w} h(t^N)^m \bigg(\sum_{n=mw+r}^\infty b_{n}t^{(n-mw-r)N/w} -\varphi(c) \sum_{n\geq m+r/w}a_{n} t^{(n-m-r/w)N}\bigg).
\end{align}

Consider the case where $v$ is finite.   From (\ref{E:texpansion with cancel}), we have $\varphi(P)-\varphi(c)=t^{rN/w} g(t)$, where $g$ is a power series with coefficients in $\ZZ[\zeta_N]$.  Therefore, $|\varphi(P)-\varphi(c)|_v \leq |t|_v^{rN/w} = |j(P)|_v^{-r/w}$.   We obtain $|\varphi(P)-\varphi(c)|_v \leq |j(P)|_v^{-1/w}$ by taking $r=1$.  This proves (\ref{T:nouveau varphi D}) in the case that $v$ is finite.

Now suppose that $v$ is infinite and take $r=1$.  Starting with (\ref{E:varphiP raw t}) and taking absolute values gives
\begin{align}\label{E:abs inf version}
|\varphi(P)-\varphi(c)|_v\leq u^{-m} |h(t^N)|_v^m\bigg(\sum_{n=mw+1}^\infty |b_{n}|_v u^{n/w} +|\varphi(c)|_v \sum_{n= m+1}^\infty |a_{n}| u^{n}\bigg),
\end{align}
where $u:=|t|_v^N\leq e^{-\pi\sqrt{3}}$.  Using our bound $|b_n|_v \leq \beta (n/w)^{24m}$ for $n\geq mw$, Lemma~\ref{L:sum of cn}(\ref{L:sum of cn ii}) with $B=mw+1$ implies that  
\[
u^{-m} \sum_{n=mw+1}^\infty |b_{n}|_vu^{n/w} \leq \beta\cdot u^{1/w}\cdot 231.6 w (5m)^{24m+1}.
\]  
By Lemma~\ref{L:Deltam sum bounds}(\ref{L:Deltam sum bounds iii}) with $B=m+1$, we have $u^{-m}\sum_{n=m+1}^\infty |a_n| u^n \leq 465 u(2m)^{6m+1}$ and hence
\[
u^{-m} |\varphi(c)|_v \sum_{n=m+1}^\infty |a_n| u^n \leq \beta\cdot u\cdot 465 \cdot 2^{6m+1} m^{30m+1}
\]
by Lemma~\ref{L:varphi basic}(\ref{L:varphi basic iv}).  We have $|h(t^N)|_v <1.1104$ by (\ref{E:easy h}).  Combining the above bounds with (\ref{E:abs inf version}) gives
\begin{align*}
|\varphi(P)-\varphi(c)|_v &\leq \beta\cdot 1.1104^m( u^{1/w}\cdot 231.6 w (5m)^{24m+1} + u\cdot 465 \cdot 2^{6m+1}  m^{30m+1})\\
&\leq u^{1/w} \beta  w \cdot 1.1104^m( 231.6  (5m)^{24m+1} +  465 \cdot 2^{6m+1}  m^{30m+1})\\
&\leq u^{1/w} \beta  w \cdot 1.1104^m( 231.6  \cdot 5^{24m+1} +  465 \cdot 2^{6m+1}  ) m^{30m+1}\\
&= u^{1/w} \beta  w \cdot 1.1104^m 1158\cdot 5^{24m}( 1 +  \tfrac{465}{231.6}\cdot \tfrac{2}{5} \cdot (\tfrac{2^6}{5^{24}})^m  ) m^{30m+1}\\
&\leq u^{1/w} \beta C_0/2,
\end{align*}
where $C_0:=N \cdot 2\cdot 1159\cdot 5.022^{24m} m^{30m+1}$.  Since $u^{1/w}\leq 1$, we deduce that $|\varphi(P)-\varphi(c)|_v\leq  \beta  C_0/2$.   Now assume further that $|j(P)|_v>3500$. We have $u=|t|_v^N\leq 2|j(P)|_v^{-1}$ by Lemma~\ref{L:final t} and our of choice of $t$.  So $u^{1/w}\leq (2|j(P)|_v^{-1})^{1/w}\leq 2 |j(P)|_v^{-1/w}$ and hence $|\varphi(P)-\varphi(c)|_v\leq  |j(P)|_v^{-1/w} \cdot \beta C_0$.

To complete the proof of (\ref{T:nouveau varphi D}) in the case that $v$ is infinite, it suffices to show that $C_0\leq C$.  Since $m\leq \tfrac{1}{24}N^3$ by Lemma~\ref{L:trivial bounds}, we have 
\[
C_0\leq 2\cdot 1159\cdot 5.022^{24m}  (\tfrac{1}{24})^{30m+1} N^{90m+4}\leq 96.6 \cdot 0.1^{24m} N^{90m+4} = C.\\
\]

It remains to prove (\ref{T:nouveau varphi E}).  Let $K$ be any number field with $K_G\subseteq K \subseteq L$ and let $\Sigma'$ be a $\Gal_K$-orbit of $\Sigma$.   We keep notation as above with a fixed cusp $c\in \Sigma'\subseteq\Sigma$ and we take $r=\ord_c(\varphi-\varphi(c))$.   We further assume that $\varphi(c)\in K$; if no such $c$ exists then (\ref{T:nouveau varphi E}) holds trivially with $\xi=1$.  Since the divisor of $\varphi-\varphi(c)$ has degree $0$, we have $r\leq m\mu$ by Lemma~\ref{L:varphi basic}(\ref{L:varphi basic v}). 

We have $f*A-\varphi(c)\Delta^2=\gamma q_w^{mw+r}+\cdots$ with $\gamma\in \ZZ[\zeta_N]$ nonzero.  From (\ref{E:qexpansion with cancel}), we find that 
\begin{align} \label{E:gamma b a}
\gamma=
\begin{cases}
 b_{mw+r}-\varphi(c)a_{m+r/w} & \text{ if $w|r$},\\
  b_{mw+r} & \text{ if $w\nmid r$.}
\end{cases}
\end{align}
The rational function $h:=(\varphi-\varphi(c))^w j^r$ lies in $K(X_G)$ and we have $\ord_c h = wr-rw=0$.   Moreover, $h(c)=\gamma^w$.   Since $h$ is defined over $K$, we have $\gamma^w=h(c)\in K(c)$.  Define
\[
\xi:=N_{K(c)/K}(\gamma^w) \in K.
\]
We have $\xi\in \OO_K$ since $\gamma$ is integral.  We have $\xi\neq 0$ since $\gamma\neq 0$.

We will prove that (\ref{T:nouveau varphi E}) holds with this particular $\xi$.  Since $\varphi(c)$ lies in $K$, by assumption, and $\varphi$ is defined over $K$, we deduce that $\varphi(c')=\varphi(c)\in K$ for all $c'\in \Sigma'$.   We also have $K(c')=K(c)$ and $w_{c'}=w_c$ for all $c\in \Sigma'$ since $\Sigma'$ has a transitive $\Gal(L/K)$-action and $K(c)/K$ is abelian.  Therefore, $h(c)=\gamma^w$ and $\xi=N_{K(c)/K}(\gamma^w)$ are independent of the choice of $c\in \Sigma'$.   So for the rest of the proof we can work with our fixed cusp $c$.

We now find upper bounds for $|\gamma|_v$ for some places $v$ of $L$.  Define $C':=22.16  N^{144m+7} 0.024^{24m}$.

\begin{lemma} \label{L:gamma bound infinite}
For any infinite place $v$ of $L$, we have $|\gamma|_v \leq \beta C'$.
\end{lemma}
\begin{proof}
Take any infinite place $v$ of $L$.  By Theorem~\ref{T:modular form of weight k}(\ref{T:modular form of weight k c}) and $r\leq m\mu$, we have
\[
|b_{mw+r}|_v\leq \beta (m+r/w)^{24m} \leq \beta m^{24m} (1+\mu)^{24m}.
\]
If $w$ divides $r$, Lemma~\ref{L:Deltam sum bounds}(\ref{L:Deltam sum bounds i}) and $r\leq m\mu$ implies that $|a_{m+r/w}|\leq 2(m+r/w)^{6m} \leq 2m^{6m}(1+\mu)^{6m}$ and hence
\[
|\varphi(c) a_{m+r/w}|_v \leq  \beta  \cdot 2m^{30m} (1+\mu)^{6m}
\]
by Lemma~\ref{L:varphi basic}(\ref{L:varphi basic iv}).   Using these bounds with (\ref{E:gamma b a}), we obtain 
\begin{align*}
|\gamma|_v & \leq \beta m^{24m} (1+\mu)^{24m} + \beta  \cdot 2m^{30m} (1+\mu)^{6m}.
\end{align*}
Using the bounds on $m$ and $\mu+1$ from Lemma~\ref{L:trivial bounds}, we obtain
\begin{align*}
|\gamma|_v & \leq \beta ( (\tfrac{1}{24}\cdot \tfrac{29}{54})^{24m} N^{144m} + 2  (\tfrac{1}{24})^{30m} (\tfrac{29}{54})^{6m} N^{108m})\\
&= \beta N^{144m} (\tfrac{1}{24}\cdot \tfrac{29}{54})^{24m} (1+ 2 (\tfrac{1}{24})^{6m} (\tfrac{54}{29})^{18m} N^{-36m})\\
&\leq \beta N^{144m} (\tfrac{1}{24}\cdot \tfrac{29}{54})^{24m} (1+ 2 (\tfrac{1}{24}(\tfrac{54}{29})^3 3^{-6})^{6m}),
\end{align*}
where in the last inequality we have used that $N\geq 3$.  Since $\tfrac{1}{24}(\tfrac{54}{29})^3 3^{-6}\leq 1$, we have $|\gamma|_v \leq \beta N^{144m} (\tfrac{1}{24}\cdot \tfrac{29}{54})^{24m} (1+ 2 (\tfrac{1}{24}(\tfrac{54}{29})^3 3^{-6})^{6})$.   It is now easy to verify that $|\gamma|_v\leq \beta C'$.
\end{proof}

\begin{lemma} \label{L:almost there xi}
Consider any point $P \in Y_G(K)$ that satisfies $\varphi(P)=\varphi(c)$ and $P\in \Omega_{c,v}$ for some place $v$ of $L$.  Further suppose that $|j(P)|_v>3500$ when $v$ is infinite.  Then 
\[
|\gamma|_v \leq
\begin{cases}
|j(P)|_v^{-1} & \text{ if $v$ is finite,}\\
|j(P)|_v^{-1} \cdot \beta C' & \text{ if $v$ is infinite and $|j(P)|_v>3500$.}
\end{cases}
\]
\end{lemma}
\begin{proof}
Using equation (\ref{E:varphiP raw t}) and $\varphi(P)-\varphi(c)=0$, we obtain
\begin{align} \label{E:gamma unsolved}
0=\gamma + t^{-(mw+r)N/w} \sum_{n=mw+r+1}^\infty b_{n}t^{nN/w} -\varphi(c)t^{-(mw+r)N/w} \sum_{n> m+r/w}a_{n} t^{nN},
\end{align}
with $t\in \Lbar_v$ satisfying $|t|_v<1$.  If $v$ is finite, we also have $|j(P)|_v=|t|_v^{-N}$.  If $v$ is infinite, we also have $|t|_v\leq e^{-\pi\sqrt{3}/N}$ and $|t|_v^N \leq 2|j(P)|_v^{-1}$.  In particular, $\gamma=t^{N/w}g(t)$ in $\Lbar_v$ for some power series $g(x)$ with coefficients in $\ZZ[\zeta_N]$.   So when $v$ is finite, we have $|\gamma|_v=|t|_v^{N/w}|g(t)|_v \leq |t|_v^{N/w}=|j(P)|_v^{-1/w}$.   

We can now assume that $v$ is infinite.  Define $u=|t|_v^N$; we have $u\leq e^{-\pi\sqrt{3}}$.
Solving for $\gamma$ in (\ref{E:gamma unsolved}) and taking absolute values, we obtain
\begin{align} \label{E: gamma s1 s2}
|\gamma|_v \leq u^{-(mw+r)/w}\, s_1 +|\varphi(c)|_v u^{-(mw+r)/w} \,s_2,
\end{align}
where $s_1:=\sum_{n=mw+r+1}^\infty |b_{n}|_v u^{n/w}$ and $s_2:=\sum_{n> m+r/w}|a_{n}| u^{n}$.   Using Lemma~\ref{L:sum of cn} with $B=mw+r+1$ and our bound $|b_n|\leq \beta \max\{1,(n/w)^{24m}\}$, we have
\begin{align*}
u^{-(mw+r)/w} s_1 &\leq   \beta\cdot  231.6wu^{1/w}\max\{(mw+r+1)/w,5m\}^{24m+1}\\
&\leq  u^{1/w} \beta\cdot  231.6Nm^{24m+1}\max\{\mu+2,5\}^{24m+1},
\end{align*}
where the last inequality uses $1\leq w\leq N$ and $r\leq m\mu$.  Using the bounds on $m$ and $\mu+2$ from Lemma~\ref{L:trivial bounds}, we obtain
\begin{align*}
u^{-(mw+r)/w} s_1 &\leq u^{1/w}\beta\cdot 231.6  N^{144m+7}(\tfrac{1}{24}\cdot \tfrac{31}{54})^{24m+1} \leq u^{1/w}\beta\cdot 5.54 N^{144m+7} 0.024^{24m}.
\end{align*}

Let $B$ be the smallest integer for which $B>m+r/w$. By Lemma~\ref{L:Deltam sum bounds}, we have
\begin{align*}
s_2
& \leq 465u^{B} \max\{m+r/w,2m\}^{6m+1}\\
&\leq 465u^{m+r/w+1/w} m^{6m+1} \max\{1+\mu,2\}^{6m+1},
\end{align*}
where in the last inequality we have used that $u<1$, $B\geq m+r/w+1/w$, and $r/w\leq r \leq m\mu$.   So by Lemma~\ref{L:varphi basic}(\ref{L:varphi basic iv}),
\[
 |\varphi(c)|_v u^{-(mw+r)/w} s_2 \leq u^{1/w} \beta \cdot 465 m^{30m+1} \max\{1+\mu,2\}^{6m+1}.
\]
Using the bounds on $m$ and $\mu+1$ from Lemma~\ref{L:trivial bounds}, we obtain
\begin{align*}
 |\varphi(c)|_v u^{-(mw+r)/w} s_2 &\leq u^{1/w} \beta\cdot 465 N^{108m+6} (\tfrac{1}{24})^{30m+1}  (\tfrac{29}{54})^{6m+1}\\
 & \leq u^{1/w} \beta\cdot 10.41 N^{108m+6} 0.0162^{24m}.
\end{align*}
In particular, $ |\varphi(c)|_v u^{-(mw+r)/w} s_2 \leq u^{1/w}\beta\cdot 5.54 N^{144m+7} 0.024^{24m}$ since $N\geq 3$.

 Thus from (\ref{E: gamma s1 s2}) and the above bounds, we have
 \begin{align*}
 |\gamma|_v &\leq 2\cdot u^{1/w}\beta\cdot 5.54 N^{144m+7} 0.024^{24m} =u^{1/w} \beta C'/2.
 \end{align*}
 Since $u^{1/w}=|t|_v^{N/w} \leq (2|j(P)|_v^{-1})^{1/w}\leq 2 |j(P)|_v^{1/w}$, we conclude that $|\gamma|_v \leq |j(P)|_v^{1/w} \beta C'$.
\end{proof}

We have $\gamma^w \in K(c)$.  For any place $v$ of $L$, we have
\[
|\xi|_v = \prod_{\sigma\in \Gal(K(c)/K)} |\sigma(\gamma^w)|_v.
\]
Since the action of $\Gal(K(c)/K)$ on $\Sigma'$ is transitive and $c\in \Sigma'$, we have $|\Gal(K(c)/K)|=|\Sigma'|$.

Suppose that $v$ is infinite.  We have $|\sigma(\gamma^w)|_v \leq (\beta C')^w$ for all $\sigma\in \Gal(K(c)/K)$ by Lemma~\ref{L:gamma bound infinite}; note that the lemma holds for an arbitrary infinite place.  Therefore, 
\[
|\xi|_v \leq (\beta C')^{w|\Gal(K(c)/K)|}=(\beta C')^{w|\Sigma'|}.
\]
Now suppose further that $|j(P)|_v>3500$ and that there is a point $P\in Y_G(K)$ for which $\varphi(P)=\varphi(c)$ and $P\in \Omega_{c,v}$.   By Lemma~\ref{L:almost there xi}, we have $|\gamma^w|_v\leq |j(P)|_v^{-w} (\beta C')^w$.  Therefore,
\[
|\xi|_v \leq (\beta C')^{w(|\Gal(K(c)/K)|-1)} |\gamma^w|_v \leq  |j(P)|_v^{-w} (\beta C')^{w|\Gal(K(c)/K)|}=|j(P)|_v^{-w} (\beta C')^{w|\Sigma'|}.
\]

Finally suppose that $v$ is finite and that there is a point $P\in Y_G(K)$ for which $\varphi(P)=\varphi(c)$ and $P\in \Omega_{c,v}$.   Since $\gamma$ is integral, we have $|\sigma(\gamma^w)|_v \leq 1$ for all $\sigma\in \Gal(K(c)/K)$.   Therefore, $|\xi|_v\leq |\gamma^w|_v$.  By Lemma~\ref{L:almost there xi}, we conclude that $|\xi|_v \leq |j(P)|_v^{-w}$.


\begin{bibdiv}
\begin{biblist}

\bib{MR1336325}{article}{
   author={Bilu, Yuri},
   title={Effective analysis of integral points on algebraic curves},
   journal={Israel J. Math.},
   volume={90},
   date={1995},
   number={1-3},
   pages={235--252},
   issn={0021-2172},
   review={\MR{1336325}},
   %doi={10.1007/BF02783215},
}

\bib{BiluParent}{article}{
   author={Bilu, Yuri},
   author={Parent, Pierre},
   title={Runge's method and modular curves},
   journal={Int. Math. Res. Not. IMRN},
   date={2011},
   number={9},
   pages={1997--2027},
   issn={1073-7928},
   review={\MR{2806555}},
   %doi={10.1093/imrn/rnq141},
}


\bib{MR2216774}{book}{
   author={Bombieri, Enrico},
   author={Gubler, Walter},
   title={Heights in Diophantine geometry},
   series={New Mathematical Monographs},
   volume={4},
   publisher={Cambridge University Press, Cambridge},
   date={2006},
   pages={xvi+652},
   isbn={978-0-521-84615-8},
   isbn={0-521-84615-3},
   review={\MR{2216774}},
   %doi={10.1017/CBO9780511542879},
}



\bib{MR3705252}{article}{
   author={Brunault, Fran\c{c}ois},
   title={R\'{e}gulateurs modulaires explicites via la m\'{e}thode de
   Rogers-Zudilin},
   language={French, with English and French summaries},
   journal={Compos. Math.},
   volume={153},
   date={2017},
   number={6},
   pages={1119--1152},
   issn={0010-437X},
   review={\MR{3705252}},
   %doi={10.1112/S0010437X17007023},
}

\bib{BN2019}{article}{
  author={Brunault, Fran\c{c}ois},
  author={Neururer, Michael},
  title={Fourier expansions at cusps},
  journal={The Ramanujan Journal}
  date={2019},
%  doi={https://doi.org/10.1007/s11139-019-00178-5}
}

\bib{MR4450724}{article}{
   author={Cai, Yulin},
   title={An explicit bound of integral points on modular curves},
   journal={Commun. Math.},
   volume={30},
   date={2022},
   number={1},
   pages={161--174},
   issn={1804-1388},
   review={\MR{4450724}},
 %  doi={10.46298/cm.9389},
}

\bib{MR2016709}{article}{
      author={Cummins, C.~J.},
      author={Pauli, S.},
       title={Congruence subgroups of {${\rm PSL}(2,{\Bbb Z})$} of genus less
  than or equal to 24},
        date={2003},
        ISSN={1058-6458},
     journal={Experiment. Math.},
      volume={12},
      number={2},
       pages={243\ndash 255},
%         url={http://projecteuclid.org/getRecord?id=euclid.em/1067634734},
%      review={\MR{MR2016709 (2004i:11037)}},
}

\bib{MR337993}{article}{
   author={Deligne, P.},
   author={Rapoport, M.},
   title={Les sch\'{e}mas de modules de courbes elliptiques},
   language={French},
   conference={
      title={Modular functions of one variable, II},
      address={Proc. Internat. Summer School, Univ. Antwerp, Antwerp},
      date={1972},
   },
   book={
      series={Lecture Notes in Math., Vol. 349},
      publisher={Springer, Berlin-New York},
   },
   date={1973},
   pages={143--316},
   review={\MR{337993}},
}

\bib{MR1109353}{article}{
   author={Faltings, Gerd},
   title={Diophantine approximation on abelian varieties},
   journal={Ann. of Math. (2)},
   volume={133},
   date={1991},
   number={3},
   pages={549--576},
   issn={0003-486X},
   review={\MR{1109353}},
   %doi={10.2307/2944319},
}


\bib{MR2104361}{article}{
   author={Kato, Kazuya},
   title={$p$-adic Hodge theory and values of zeta functions of modular
   forms},
   language={English, with English and French summaries},
   note={Cohomologies $p$-adiques et applications arithm\'{e}tiques. III},
   journal={Ast\'{e}risque},
   number={295},
   date={2004},
   pages={ix, 117--290},
   issn={0303-1179},
   review={\MR{2104361}},
}

\bib{MR0447119}{article}{
   author={Katz, Nicholas M.},
   title={$p$-adic properties of modular schemes and modular forms},
   conference={
      title={Modular functions of one variable, III},
      address={Proc. Internat. Summer School, Univ. Antwerp, Antwerp},
      date={1972},
   },
   book={
      publisher={Springer, Berlin},
   },
   date={1973},
   pages={69--190. Lecture Notes in Mathematics, Vol. 350},
%   review={\MR{0447119}},
}

\bib{MR2904927}{article}{
   author={Khuri-Makdisi, Kamal},
   title={Moduli interpretation of Eisenstein series},
   journal={Int. J. Number Theory},
   volume={8},
   date={2012},
   number={3},
   pages={715--748},
   issn={1793-0421},
   review={\MR{2904927}},
   %doi={10.1142/S1793042112500418},
}

\bib{MR1289008}{article}{
   author={Kooman, Robert Jan},
   title={Faltings's version of Siegel's lemma},
   conference={
      title={Diophantine approximation and abelian varieties},
      address={Soesterberg},
      date={1992},
   },
   book={
      series={Lecture Notes in Math.},
      volume={1566},
      publisher={Springer, Berlin},
   },
   date={1993},
   pages={93--96},
   review={\MR{1289008}},
   %doi={10.1007/978-3-540-48208-6_10},
}

\bib{MR3917917}{article}{
   author={Le Fourn, Samuel},
   title={A tubular variant of Runge's method in all dimensions, with
   applications to integral points on Siegel modular varieties},
   journal={Algebra Number Theory},
   volume={13},
   date={2019},
   number={1},
   pages={159--209},
   issn={1937-0652},
   review={\MR{3917917}},
   %doi={10.2140/ant.2019.13.159},
}

\bib{MR2471957}{article}{
   author={Rouse, Jeremy},
   title={Bounds for the coefficients of powers of the $\Delta$-function},
   journal={Bull. Lond. Math. Soc.},
   volume={40},
   date={2008},
   number={6},
   pages={1081--1090},
   issn={0024-6093},
   review={\MR{2471957}},
   %doi={10.1112/blms/bdn094},
}

\bib{MR2459823}{book}{
   author={Schoof, Ren\'{e}},
   title={Catalan's conjecture},
   series={Universitext},
   publisher={Springer-Verlag London, Ltd., London},
   date={2008},
   pages={x+124},
   isbn={978-1-84800-184-8},
   review={\MR{2459823}},
}

\bib{MR3250041}{article}{
   author={Sha, Min},
   title={Bounding the $j$-invariant of integral points on modular curves},
   journal={Int. Math. Res. Not. IMRN},
   date={2014},
   number={16},
   pages={4492--4520},
   issn={1073-7928},
   review={\MR{3250041}},
   %doi={10.1093/imrn/rnt085},
}

\bib{MR1291394}{book}{
   author={Shimura, Goro},
   title={Introduction to the arithmetic theory of automorphic functions},
   series={Publications of the Mathematical Society of Japan},
   volume={11},
   note={Reprint of the 1971 original;
   Kan\^{o} Memorial Lectures, 1},
   publisher={Princeton University Press, Princeton, NJ},
   date={1994},
   pages={xiv+271},
%   isbn={0-691-08092-5},
%   review={\MR{1291394}},
}


\bib{MR1312368}{book}{
   author={Silverman, Joseph H.},
   title={Advanced topics in the arithmetic of elliptic curves},
   series={Graduate Texts in Mathematics},
   volume={151},
   publisher={Springer-Verlag, New York},
   date={1994},
   pages={xiv+525},
   isbn={0-387-94328-5},
   review={\MR{1312368}},
   %doi={10.1007/978-1-4612-0851-8},
}


\bib{OpenImage}{article}{
	author={Zywina, David},
	title={Explicit open images for elliptic curves over $\QQ$},
	date={2024},
	note={\href{https://arxiv.org/abs/2206.14959}{arXiv:2206.14959} [math.NT]},
}




\end{biblist}
\end{bibdiv}
\end{document}